\documentclass[reqno]{amsart}
\usepackage[left=1in,right=1in,top=1in,bottom=1in]{geometry}
\usepackage{tikz}
\usetikzlibrary{shapes,decorations,calc,arrows}
\usepackage[foot]{amsaddr}
\usepackage{amsthm}
\usepackage{amscd}
\usepackage{amsfonts}
\usepackage{amsmath}
\usepackage{amssymb}
\usepackage{mathrsfs}
\usepackage{multirow}
\usepackage{verbatim}
\usepackage{url}
\usepackage[hidelinks]{hyperref}
\usepackage{graphicx}
\usepackage{cite}
\usepackage{fancyhdr}
\usepackage{yfonts}
\usepackage{setspace}
\usepackage{titlesec}
\usepackage{enumitem}

\usepackage{ dsfont }

\usepackage{tocloft} % fixes some weird TOC spacing

\titleformat{\section}[hang]
{\normalfont\Large\bfseries}
{\thesection.}{0.5em}{}

\titlespacing*{\section}{0pc}{2pc}{0.25pc}

\titleformat{\subsection}[runin]
{\normalfont\large\bfseries}
{\thesubsection}{0.5em}{}

\titlespacing{\subsection}{0pc}{1.5pc}{0.5pc}

\newcommand{\supp}{\text{supp}}
\newcommand{\Aut}{\text{Aut}}

\newcommand{\N}{\mathbb{N}}
\newcommand{\Z}{\mathbb{Z}}
\newcommand{\Q}{\mathbb{Q}}
\newcommand{\R}{\mathbb{R}}
\newcommand{\C}{\mathbb{C}}
\renewcommand{\H}{\mathcal{H}}
\newcommand{\B}{\mathcal{B}}

\newcommand{\K}{\mathcal{K}}

\newcommand{\vphi}{\varphi}

\newcommand{\Tr}{\text{Tr}}

\newcommand{\<}{\left\langle}
\renewcommand{\>}{\right\rangle}

\newcommand{\dom}{\text{dom}}

\newcommand{\Ad}[1]{\text{Ad}\left(#1\right)}

\newcommand{\sltimes}[1]{\mathbin{_{#1}\ltimes}}

\newcommand{\E}{\mathcal{E}}
\newcommand{\Ind}[1]{\operatorname{Ind}#1}

\newcommand{\Sd}{\operatorname{Sd}}
\renewcommand{\S}{\operatorname{S}}
\newtheorem{thm}{Theorem}[section]
\newtheorem{thmalpha}{Theorem}

\newtheorem{prop}[thm]{Proposition}
\newtheorem{lem}[thm]{Lemma}
\newtheorem*{lem*}{Lemma}
\newtheorem{cor}[thm]{Corollary}

\theoremstyle{definition}
\newtheorem{defi}[thm]{Definition}
\newtheorem{ex}[thm]{Example}
\newtheorem{exs}[thm]{Examples}
\newtheorem{rem}[thm]{Remark}

\title{\textbf{Almost Unimodular Groups}}
\author{Aldo Garcia Guinto$^\circ$}
\address{$^\circ$Department of Mathematics, Michigan State University\hfill \url{garci575@msu.edu}}
\author{Brent Nelson$^\bullet$}
\address{$^\bullet$Department of Mathematics, Michigan State University \hfill \url{brent@math.msu.edu}}
\date{}

\begin{document}

\begin{abstract}
We show that a locally compact group has open unimodular part if and only if the Plancherel weight on its group von Neumann algebra is almost periodic. We call such groups \emph{almost unimodular}. The almost periodicity of the Plancherel weight allows one to define a Murray--von Neumann dimension for certain Hilbert space modules over the group von Neumann algebra, and we show that for finite covolume subgroups this dimension scales according to the covolume. Using this we obtain a generalization of the Atiyah--Schmid formula in the setting of second countable almost unimodular groups with finite covolume subgroups. Additionally, for the class of almost unimodular groups we present many examples, establish a number of permanence properties, and show that the formal degrees of irreducible and factorial square integrable representations are well behaved.
\end{abstract}

\maketitle
%\tableofcontents

%%%%%%%%%%%%%%%%%%%%%%%%%%%%%
%       Introduction        %
%%%%%%%%%%%%%%%%%%%%%%%%%%%%%

\section*{Introduction}

For any locally compact group $G$, its modular function $\Delta_G\colon G\to \R_+$ determines a continuous injective homomorphism
    \begin{align}\label{eqn:quotient_map}
        G/\ker{\Delta_G}\ni s \ker{\Delta_G}\mapsto \Delta_G(s) \in \R_+,
    \end{align}
where the domain is equipped with the quotient topology and the codomain with its usual topology. It follows from classical results that either this map is an isomorphism of locally compact groups onto $\R_+$, or $\ker{\Delta_G}$ is open in $G$ (see Lemma~\ref{lem:dichotomy}). In this article, we provide an operator algebraic perspective on this dichotomy by characterizing groups in the latter class as those with almost periodic Plancherel weights.

The \emph{Plancherel weight} associated to a left Haar measure $\mu_G$ on $G$ is a faithful normal semifinite weight $\varphi_G$ on the group von Neumann algebra $L(G)$ that satisfies
    \[
        \varphi_G(\lambda_G(f)^* \lambda_G(g)) = \int_{G} \overline{f(s)} g(s)\ d\mu_G(s)
    \]
for all continuous compactly supported functions $f,g\in C_c(G)$ (more generally for any left convolvers in $L^2(G)$). Here we denote
    \[
        \lambda_G(f):= \int_G \lambda_G(s) f(s)\ d\mu_G(s)
    \]
for any $f\in L^1(G)$, where $G\ni s\mapsto \lambda_G(s) \in \mathcal{U}(L^2(G))$ is the left regular representation. The above equality yields an isomorphism of Hilbert spaces $L^2(L(G),\varphi_G)\cong L^2(G)$ that additionally carries the modular operator $\Delta_{\varphi_G}$ of $\varphi_G$ to the modular function $\Delta_G$ of $G$ acting by pointwise multiplication. It follows that the modular automorphism group $\sigma^{\varphi_G}\colon \R\curvearrowright L(G)$ is completely determined by $\Delta_G$:
    \[
        \sigma_t^{\varphi_G}(\lambda_G(f)) = \int_G \Delta_G(s)^{it} \lambda_G(s) f(s)\ d\mu_G(s) \qquad t\in \R.
    \]
This strong correspondence means that modular properties of $\varphi_G$ can be characterized purely in terms of the group $G$ (e.g. $\varphi_G$ is tracial if and only if $\Delta_G\equiv 1$). Our first main result is a particular instance of this.

\begin{thmalpha}[{Theorem~\ref{thm:when_is_plancherel_weight_almost_periodic}}]\label{introthm:A}
For a locally compact group $G$, the following are equivalent:
    \begin{enumerate}[label=(\roman*)]
        \item $\ker{\Delta_G}$ is open;
        \item $\varphi_G$ is strictly semifinite;
        \item $\varphi_G$ is almost periodic;
        \item $\Delta_G$ viewed as an operator affiliated with $L^\infty(G) \subset B(L^2(G))$ has non-empty point spectrum.
    \end{enumerate}
In this case, the point spectrum of $\Delta_G$ is $\Delta_G(G)$, and one has $L(\ker{\Delta_G})\cong L(G)^{\varphi_G}$ where the restriction $\varphi_G|_{L(\ker{\Delta_G})}$ is the Plancherel weight corresponding to the restriction of the left Haar measure $\mu_G$ to $\ker{\Delta_G}$.
\end{thmalpha}

\noindent Here \emph{strictly semifinite} means the restriction of the weight $\varphi_G$ to its centralizer subalgebra $L(G)^{\varphi_G}$ is still semifinite, and \emph{almost periodic}---a definition due to Connes \cite{Con72,Con74}---means the modular operator $\Delta_{\varphi_G}$ is diagonalizable. In general, the former is implied by the latter, but in the case of group von Neumann algebras the above theorem shows the converse also holds. In light of these equivalences, we call a group $G$ satisfying any of these conditions an \textbf{almost unimodular group}. In this paper, we study the properties of these groups and their von Neumann algebras.

The terminology ``almost unimodular group'' of course alludes to the almost periodicity of the Plancherel weight $\varphi_G$, but it is also meant to imply that the unimodular part $\ker{\Delta_G}$ strongly influences the structure of $G$ itself. Indeed, the openness of $\ker{\Delta_G}$ implies that
    \begin{align}\label{eqn:au_ses}
        1 \to \ker{\Delta_G} \to G \to \Delta_G(G) \to 1
    \end{align}
is a short exact sequence of locally compact groups, where $\Delta_G(G)$ is given the discrete topology. Thus, for example, $G$ is amenable if and only if $\ker{\Delta_G}$ is amenable (see \cite[Theorem 4.B]{Neu29} and also \cite[Proposition G.2.2]{BdlHV08}), and $G$ has the Haagerup property if and only if $\ker{\Delta_G}$ has the Haagerup property (see \cite[Proposition 2.5]{Jol00} or \cite[Proposition 6.1.5]{CCJJV01}). The influence of the unimodular part also manifests in the representation theory of $G$ (see Theorems~\ref{thm:sq_int_irreps}, \ref{thm:formal_degree_is_almost_periodic}, and \ref{thm:sq_int_factor_reps}), which we further discuss below. We also remark that this parallels the situation for von Neumann algebras: if $\varphi$ is an almost periodic weight on a von Neumann algebra $M$, then $(M^\varphi)'\cap M =(M^\varphi)'\cap M^\varphi$ (see \cite[Theorem 10]{Con72}). This can be interpreted as saying that the centralizer $M^\varphi$ is large relative to $M$, and, in particular, $M^\varphi$ being a factor guarantees that $M$ is also a factor.

The class of almost unimodular groups of course includes all \emph{unimodular} groups, but also all totally disconnected groups (see Example~\ref{exs:first_examples}.(\ref{ex:totally_disconnected})) and consequently all automorphism groups of connected locally finite graphs (see Example~\ref{exs:first_examples}.(\ref{ex:graph_automorphisms})). Further concrete non-unimodular examples can been found in Examples~\ref{exs:first_examples} and \ref{ex:interesting_group_vNas}. In Section~\ref{sec:permanence_properties} we study the permanence properties of this class and determine precisely which subgroups, quotients, continuous cocycle semidirect products, and fiber products yield almost unimodular groups (see Propositions~\ref{prop:almost_unimodular_subgroups}, \ref{prop:quotient_groups}, \ref{prop:cocycle_semidirect_products}, and \ref{prop:fiber_products}). Additionally, we show that the class of almost unimodular groups is the smallest class of locally compact groups that is closed under extensions by discrete groups and contains all unimodular groups (see Corollary~\ref{cor:class_of_almost_unimod_groups}).

In Section~\ref{sec:representation_theory}, we study the behavior of irreducible and factorial square integrable representations of almost unimodular groups. Such representations were studied for general (non-unimodular) locally compact groups by Duflo and Moore in \cite{DM76} and by Moore in \cite{Moo77}, respectively, where the notion of the so-called \emph{formal degree} was extended beyond the unimodular case. For an irreducible square integrable representation $(\pi,\H)$, this is given by an unbounded operator $D$ on $\H$, and for a factorial square integrable representation this is given by a faithful normal semifinite weight $\psi$ on $\pi(G)''$. For unimodular groups, these reduce to the identity operator and a tracial weight, respectively. We show that almost unimodularity tames these objects while still allowing them to be non-trivial: $D$ is diagonalizable (see Theorem~\ref{thm:sq_int_irreps}) and $\psi$ is almost periodic (see Theorem~\ref{thm:formal_degree_is_almost_periodic}). Almost unimodularity also forges connections between such representations of $G$ and representations of its unimodular part $\ker{\Delta_G}$. One example of this is that every irreducible square integrable representation of $G$ is induced by an irreducible square integrable representation of $\ker{\Delta_G}$ (see Theorem~\ref{thm:sq_int_irreps}). Another is the following:

\begin{thmalpha}[{Theorem~\ref{thm:sq_int_factor_reps}}]\label{introthm:B}
Let $G$ be a second countable almost unimodular group and let $(\pi_1, \H_1)$ be a factorial square integrable representation of $\ker{\Delta_G}$. Then the induced representation $\emph{Ind}_{\ker{\Delta_G}}^G(\pi_1,\H_1)$ of $G$ is factorial and square integrable.
\end{thmalpha}

\noindent Since the left regular representation of $\ker{\Delta_G}$ induces the left regular representation of $G$, the above theorem can be interpreted as a generalization of the fact mentioned above that $L(\ker{\Delta_G})\cong L(G)^{\varphi_G}$ being a factor implies $L(G)$ is factor when $\varphi_G$ is almost periodic. We also remark that a covariance condition involving the compact dual group of $\Delta_G(G)$ allows one to detect precisely which representations of $G$ (square integrable or otherwise) are induced by a representation of $\ker{\Delta_G}$ (see Theorem~\ref{thm:reps_induced_from_kernel}).

In Section~\ref{sec:group_vNas}, we return to the study of group von Neumann algebras associated to almost unimodular groups. Since the Plancherel weights of such groups are strictly semifinite, the pair $(L(G),\varphi_G)$ falls within the scope of the authors previous work \cite{GGLN25} extending Murray--von Neumann dimension to faithful normal strictly semifinite weights. There the authors defined $(L(G),\varphi_G)$-modules to be modules over the basic construction $\<L(G),e_{\varphi_G}\>$ associated to the inclusion $L(G)^{\varphi_G} \leq L(G)$. In this context, the basic construction can always be identified with $L(G)\rtimes_\alpha \Delta_G(G)\hat{\ }$, where $\alpha\colon  \Delta_G(G)\hat{\ }\curvearrowright L(G)$ is an extension of the modular automorphism group of $\sigma^{\varphi_G}$ to the compact dual group of $\Delta_G(G)$ (see Theorem~\ref{thm:basic_construction_crossed_product}), and the covariance condition involving $\Delta_G(G)\hat{\ }$ mentioned above can also be used to detect which representations of $L(G)$ admit extensions to $\<L(G),e_{\varphi_G}\>$ (see Theorem~\ref{thm:extending_reps_to_basic_construction}). So from the perspective of the group $G$, an $(L(G),\varphi_G)$-module is merely a subrepresentation of $\lambda_G^{\oplus \infty}$ that is induced from a representation of $\ker{\Delta_G}$. We also show in Section~\ref{sec:group_vNas} that factoriality of $L(\ker{\Delta_G})$ allows one to determine the type of $L(G)$ from $\Delta_G(G)$ (see Corollary~\ref{cor:factoriality}), and it forces all intermediate von Neumann algebras $L(\ker{\Delta_G})\leq P \leq L(G)$ to be of the form $P=L(H)$ for some closed intermediate group $\ker{\Delta_G}\leq H\leq G$ (see Theorem~\ref{thm:intermediate_algebras}).

Our most striking results concern how the Murray--von Neumann dimension of such $(L(G),\varphi_G)$-modules scales for subgroups $H\leq G$ with \emph{finite covolume} (which are necessarily almost unimodular; see Proposition~\ref{prop:finite_covolume_subgroups}). This means that the quotient space $G/H$ admits a finite (non-zero) $G$-invariant Radon measure $\mu_{G/H}$. For given left Haar measures $\mu_G$ and $\mu_H$ on $G$ and $H$, respectively, the measure $\mu_{G/H}$ can be normalized to satisfy a natural disintegration formula (\ref{eqn:quotient_Haar_measure_formula}), and in this case one defines the \emph{covolume} as $[\mu_G : \mu_H]:=\mu_{G/H}(G/H)$.

\begin{thmalpha}[{Theorem~\ref{thm:covolume_dimension_formula}}]\label{introthm:C}
Let $G$ be a second countable almost unimodular group with a finite covolume subgroup $H\leq G$, and suppose $\Delta_H(H)=\Delta_G(G)$.  Then there exists a unique injective, normal, unital $*$-homomorphism $\theta\colon \<L(H),e_{\varphi_H}\>\to \<L(G),e_{\varphi_G}\>$ satisfying
    \[
        \theta(\lambda_H(t)) = \lambda_G(t) \qquad t\in H, \qquad \text{ and } \qquad \theta(e_{\varphi_H}) = e_{\varphi_G}.
    \]
Moreover, if $(\pi,\H)$ is a left $(L(G),\varphi_G)$-module, then $(\pi\circ \theta, \H)$ is a left $(L(H),\varphi_H)$-module with
    \[
        \dim_{(L(H),\varphi_H)}(\pi\circ \theta, \H) = [\mu_G : \mu_H] \dim_{(L(G),\varphi_G)}(\pi,\H).
    \]
\end{thmalpha}

\noindent We note that the assumption $\Delta_H(H)=\Delta_G(G)$ can be removed. Indeed, $H\leq G$ having finite covolume implies $\Delta_H(H)$ is a finite index subgroup of $\Delta_G(G)$ (see Proposition~\ref{prop:finite_covolume_subgroups}), and so $\theta$ can be modified in this case by mapping $e_{\varphi_H}$ to an amplification of $e_{\varphi_G}$ determined by a choice of coset representatives for $\Delta_H(H)\leq \Delta_G(G)$ (see Theorem~\ref{thm:covolume_dimension_formula} for this more general statement). Also, the assumption that $G$ is second countable might not be necessary since it is primarily used to obtain a Borel section $\sigma\colon G/H\to G$ through work of Mackey from \cite{Mac52}. In the unimodular case, the analogue of the above scaling formula for \emph{abstract} modules is known to fail in general (see \cite[Example 4.6]{PetThesis}), but it has been established in a certain instances to obtain scaling formulas for $L^2$-Betti numbers (see \cite[Theorem 2]{Pet13}, \cite[Theorem B]{KPV15}, and \cite[Section 5.5]{PST18}).
    
To prove the scaling formula in Theorem~\ref{introthm:C}, one first relates the Plancherel weights $\varphi_G$ and $\varphi_H$ via an operator valued weight from $L(G)$ to $L(H)$: $\varphi_G=\varphi_H\circ T$. That such a $T$ exists is a consequence of $G/H$ admitting a $G$-invariant measure, which, moreover, provides a convenient source of elements in the domain of $T$ and allows one to compute $\varphi_H\circ T$ as an amplification of $\varphi_H$. This amplification will be infinite whenever $[G:H]$ is infinite, but the finiteness of $[\mu_G : \mu_H]$ ensures enough control to obtain the desired formula.

As an application of Theorem~\ref{introthm:C}, we generalize the Atiyah--Schmid formula to the setting of second countable almost unimodular groups.

\begin{thmalpha}[{Theorem~\ref{thm:formula}}]\label{introthm:D}
Let $G$ be a second countable almost unimodular group with a finite covolume subgroup $H\leq G$, and suppose $\Delta_H(H)=\Delta_G(G)$. Let $\theta\colon \<L(H),e_{\varphi_H}\>\to \<L(G),e_{\varphi_G}\>$ be as in Theorem~\ref{introthm:C}. If $(\pi,\H)$ is an irreducible square integrable representation of $G$, then: it is induced by an irreducible square integrable representation  $(\pi_1,\H_1)$ of $\ker{\Delta_G}$; it extends to a representation $(\widetilde{\pi},\H)$ of $\<L(G),e_{\varphi_G}\>$; and one has
    \[
        \dim_{(L(H),\varphi_H)}(\widetilde{\pi}\circ \theta, \H) = d_{\pi_1} [\mu_G: \mu_H],
    \]
where $d_{\pi_1}$ is the formal degree of $(\pi_1,\H_1)$ with respect to the restriction of $\mu_G$ to $\ker{\Delta_G}$.
\end{thmalpha}

\noindent This formula recovers \cite[Equation (3.3)]{AS77} in work of Atiyah and Schmid in the case that $H=\Gamma$ is a lattice in a (necessarily unimodular) group $G$ (see also \cite[Theorem 3.3.2]{GdlHJ89}). As with Theorem~\ref{introthm:C}, the assumption $\Delta_H(H)=\Delta_G(G)$ can be dropped, albeit at the cost of an additional scaling factor determined by a choice of coset representatives for $\Delta_H(H)\leq \Delta_G(G)$ (see Theorem~\ref{thm:formula} for this more general statement). 

By Theorem~\ref{introthm:C}, the proof of Theorem~\ref{introthm:D} is reduced to the case $H=G$, and so amounts to computing the Murray--von Neumann dimension of an irreducible square integrable representation. Notably, this computation relies very explicitly on the strict semifiniteness of $\varphi_G$, and therefore reinforces the need for $G$ to be almost unimodular through Theorem~\ref{introthm:A}.

\subsection*{Acknowledgments}
We wish to thank the following people for helpful discussions related to this work: Ionut Chifan, Rolando de Santiago, Adriana Fernández Quero, Michael Hartglass, Ben Hayes, Adrian Ioana, Srivatsav Kunnawalkam Elayavalli, and David Penneys. We especially thank David Jekel for his numerous suggestions for improving this article.  Both authors were supported by NSF grant DMS-2247047.

%%%%%%%%%%%%%%%%%%%%%%%%%%%%%
%       Preliminaries       %
%%%%%%%%%%%%%%%%%%%%%%%%%%%%%

\section{Preliminaries}

Throughout we let $G$ denote a locally compact group, which is always assumed to be Hausdorff. We will use lattice notation for the collection of closed subgroups of $G$. That is, we write $H\leq G$ to denote that $H$ is a closed subgroup of $G$, we write $H_1\vee H_2$ for the closed subgroup generated by $H_1,H_2\leq G$, etc. Similarly for a von Neumann algebra $M$ and its collection of (unital) von Neumann subalgebras and for a Hilbert space $\H$ and its collection of closed subspaces. All homomorphisms between groups are assumed to also be continuous and all isomorphisms are assumed to also be homeomorphisms. In particular, all representations on Hilbert spaces are assumed to be strongly continuous. A representation of a von Neumann algebra $M$ will always mean a normal unital $*$-homomorphism $\pi\colon M\to B(\H)$.

Throughout, $\mu_G$ will denote a \emph{left Haar measure} on $G$, which is a non-trivial, left translation invariant Radon measure. We follow the convention in \cite{Fol16} and take a \emph{Radon measure} to be a Borel measure that is finite on compact sets, outer regular on Borel sets, and inner regular on open sets. Left Haar measures always exist and are unique up to scaling. Additionally, $\mu_G(U)>0$ for all non-empty open sets $U$, $\mu_G(K)<\infty$ for all compact sets $K$, and the inner regularity of $\mu_G$ extends to all $\sigma$-finite subsets (see \cite[Section 2.2]{Fol16} or \cite[Section 11]{HR79} for further details). For $1\leq p\leq \infty$ we denote by $L^p(G)$ the $L^p$-space of $G$ with respect to left Haar measures, and for a given left Haar measure $\mu_G$ we will write $\|f\|_{L^p(\mu_G)}$ for the associated $p$-norm, whereas $\|f\|_\infty$ is the unambiguous $\infty$-norm. We also denote by $\mathcal{B}(G)$ the Borel $\sigma$-algebra on $G$.

The following lemmas are likely well-known to experts,  but will be useful in Section~\ref{sec:almost_unimodular_groups}.

\begin{lem}\label{lem:open_via_Haar_measure}
Let $G$ be a locally compact group equipped with a left Haar measure $\mu_G$. For a closed subgroup $H\leq G$, the following are equivalent:
    \begin{enumerate}[label=(\roman*)]
        \item $H$ is open;
        
        \item $\mu_G|_{\B(H)}$ is a left Haar measure on $H$;
        
        \item there exists a Borel subset $E\subset H$ with $0<\mu_G(E)<\infty$.
    \end{enumerate}
\end{lem}
\begin{proof} \textbf{(i)$\Rightarrow$(ii):} The open subsets of $H$ are open in $G$ and therefore $\mu_G|_{\B(H)}$ is inner regular on open sets. The remaining properties are inherited directly from $\mu_G$.\\

\noindent \textbf{(ii)$\Rightarrow$(iii):} Since left Haar measures are non-trivial, we must have $\mu_G(H)>0$. Inner regularity then yields a compact set $K\subset H$ with $0<\mu_G(K)< \infty$.\\

\noindent \textbf{(iii)$\Rightarrow$(i):} Suppose $0< \mu_G(E)<\infty$ for some Borel set $E\subset H$. Then $1_E \in L^1(G)\cap L^\infty(G)$ so that $f:=1_E* 1_E$ is a non-trivial, positive, continuous function with $\supp(f) \subset E\cdot E \subset H$. Hence $U:=f^{-1}(0,\infty) \subset H$ is a non-empty open set in $G$, and therefore $H$ is open.
\end{proof}

\begin{lem}\label{lem:dichotomy}
Suppose $A$ is a locally compact abelian group and $\pi\colon A\to \R_+$ is a continuous injective homomorphism. Then either $\pi$ is a isomorphism of locally compact groups or $A$ is discrete.
\end{lem}
\begin{proof}
Assume $\pi$ is not an isomorphism. If $\pi$ happens to be surjective, then we can identify $A$ with $\R_+$ as a group but equipped with a strictly finer topology. By \cite[Lemma 2.1]{Hew63}, it follows that $A$ must be discrete. Otherwise, $\pi$ is not surjective and we claim that $\R_+\setminus \pi(A)$ is dense in $\R_+$. When $\pi(A)$ itself is not dense (and therefore of the form $\lambda^\Z$ for some $0<\lambda \leq 1 $) this is immediate, and otherwise for any $t\in \R_+\setminus \pi(A)$ we have that $t\pi(A) \subset \R_+\setminus \pi(A)$ is dense. It follows that $\pi(A)$ is totally separated and hence $A$ is totally disconnected. By van Dantzig's Theorem \cite{vD36} (see also \cite[Theorem 3.1]{Wil04}), we can find a compact open subgroup $K\leq A$. Then $\pi(K)$ must be the only compact subgroup of $\R_+$, namely $\{1\}$, and the injectivity of $\pi$ implies $K=\{e\}$. In particular, $\{e\}$ is open in $A$ and thus $A$ is discrete.    
\end{proof}

Recall that the \emph{modular function} $\Delta_G\colon G\to \R_+$ is the continuous homomorphism (where $\R_+:=(0,\infty)$ has its multiplicative group structure) determined by $\mu_G(E\cdot s)= \Delta_G(s) \mu_G(E)$ for $s\in G$ and $E\in \B(G)$, where $\mu_G$ is any left Haar measure on $G$. This yields the following change of variables formulas that will be used implicitly in the sequel: 
    \[ 
        \int_Gf(st) d\mu_G(s) = \Delta_G(t)^{-1} \int_G f(s) d\mu_G(s) \qquad \text{ and } \qquad\int_G f(s^{-1}) d\mu_G(s) =  \Delta_G(s)^{-1}\int_G f(s) d\mu_G(s)
    \]
for $t\in G$ and $f\in L^1(G)$. We say $G$ is \emph{unimodular} if $\Delta_G\equiv 1$.

Let $\lambda_G,\rho_G\colon G\to B(L^2(G))$ be the \emph{left and right regular representations} of $G$:
    \[
        [\lambda_G(s)f](t)=f(s^{-1}t) \qquad \qquad [\rho_G(s)f](t)=\Delta_G(s)^{1/2}f(ts)
    \]
for $s,t\in G$ and $f\in L^2(G)$. The \emph{group von Neumann algebra} of $G$ is the von Neumann algebra generated by its left regular representation, $L(G):=\lambda_G(G)''$. We denote the von Neumann algebra generated by its right regular representation by $R(G):=\rho_G(G)''$, which we note satisfies $R(G)=L(G)'\cap B(L^2(G))$. For any $f\in L^1(G)$,
    \begin{align}\label{eqn:L1_function_operator}
        \lambda_G(f):=\int_G \lambda_G(t) f(t)\ d\mu_G(t)
    \end{align}
defines an element of $L(G)$. The mapping $f\mapsto \lambda_G(f)$ gives a $*$-homomorphism where $L^1(G)$ is equipped with the convolution
    \[
        [f*g](s) = \int_G f(t)g(t^{-1}s)\ d\mu_G(t)
    \]
and involution
    \[
        [f^\sharp](s) = \Delta_G(s)^{-1}\overline{f(s^{-1})},
    \]
and the $*$-subalgebra $\lambda_G(L^1(G))$ is dense in $L(G)$ in the strong (and weak) operator topology.

We say $f\in L^2(G)$ is a \emph{left convolver} if $f*g\in L^2(G)$ for all $g\in L^2(G)$ and there exists a constant $c>0$ such that $\|f*g\|_{L^2(\mu_G)}\leq c\|g\|_{L^2(\mu_G)}$. In this case we denote the bounded operator $g\mapsto f*g$ by $\lambda_G(f)$. Note that every $f\in L^1(G)\cap L^2(G)$ is a left convolver and that this operator agrees with the one defined in (\ref{eqn:L1_function_operator}). This additionally implies that $L(G)=\{\lambda_G(f)\colon f \text{ is a left convolver}\}''$. The \emph{Plancherel weight} associated to $\mu_G$ is then defined on $L(G)_+$ by
    \[
        \varphi_G(x^*x) := \begin{cases}
            \|f\|_{L^2(\mu_G)}^2 & \text{if }x=\lambda_G(f)\text{, with $f$ a left convolver} \\
            +\infty & \text{otherwise}
        \end{cases}.
    \]
This weight is always faithful normal and semifinite (see \cite[Section VII.3]{Tak03}).

\begin{rem}\label{rem:Plancherel_weight_is_semi-invariant_weight_Delta_G}
We claim that
    \[
        \varphi_G( \lambda_G(s) x \lambda_G(s)^*)= \Delta_G(s) \varphi_G( x) \qquad s\in G,\ x\in L(G)_+.
    \]
By definition of the Plancherel weight $\varphi_G$, it suffices to verify this for $x=\lambda_G(f)^* \lambda_G(f)$ with $f$ a left convolver. For such an $f$ and any $s\in G$, a direct computation shows that $\lambda_G(f)\lambda_G(s)^*=\lambda_G(f_s)$ where $f_s = \Delta_G(s)^{\frac12} \rho_G(s) f$ is also a left convolver. Thus
    \begin{align*}
        \varphi_G( \lambda_G(s) \lambda_G(f)^* \lambda_G(f) \lambda_G(s)^*) &= \varphi_G( \lambda_G(f_s)^* \lambda_G(f_s)^*)\\
        &= \|f_s\|_{L^2(\mu_G)}^2 \\
        &= \Delta_G(s) \| f\|_{L^2(\mu_G)}^2 = \Delta_G(s) \varphi_G(\lambda_G(f)^* \lambda_G(f)),
    \end{align*}
where we have used the definition of the Plancherel weight as well as the fact that $\rho_G(s)$ is unitary.$\hfill\blacksquare$
\end{rem}

We will assume the reader has some familiarity with modular theory for weights on von Neumann algebras and will only establish notation here. Complete details can be found in \cite[Chapter VIII]{Tak03} (see also \cite[Section 1]{GGLN25} for a quick introduction to these concepts). Given a faithful normal semifinite weight $\varphi$ on a von Neumann algebra $M$, we denote
    \begin{align*}
        \sqrt{\dom}(\varphi)&:=\{x\in M\colon \varphi(x^*x)<+\infty\}\\
        \dom(\varphi)&:=\text{span}\{x^*y\colon x,y\in \sqrt{\dom}(\varphi)\}.
    \end{align*}
We write $L^2(M,\varphi)$ for the completion of $\sqrt{\dom}(\varphi)$ with respect to the norm induced by the inner product
    \[
        \<x,y\>_\varphi:= \varphi(y^*x) \qquad x,y\in \sqrt{\dom}(\varphi).
    \]
In the case of a Plancherel weight on $L(G)$, this Hilbert space is nothing more than $L^2(G)$ since the map $\lambda_G(f)\mapsto f$ onto left convolvers extends to a unitary. The \emph{modular conjugation} and \emph{modular operator} for $\varphi$ will be denoted by $J_\varphi$ and $\Delta_\varphi$, respectively. The \emph{modular automorphism group} of $\varphi$, which we view as an action $\sigma^\varphi\colon \R\curvearrowright M$, is then defined by
    \[
        \sigma_t^\varphi(x):= \Delta_\varphi^{it} x \Delta_\varphi^{-it}.
    \]
Then the \emph{centralizer} of $\varphi$ is the fixed point subalgebra under this action and is denoted
    \[
        M^\varphi :=\{x\in M\colon \sigma_t^\varphi(x) = x \ \forall t\in \R\}.
    \]
We also recall that $x\in M^\varphi$ if and only if $xy,yx\in \dom(\varphi)$ with $\varphi(xy) = \varphi(yx)$ for all $y\in \dom(\varphi)$ (see \cite[Theorem VIII.2.6]{Tak03}). That is, $M^\varphi$ is the largest von Neumann algebra of $M$ on which $\varphi$ is tracial.
    
Lastly, we recall the two von Neumann algebraic notions that characterize almost unimodular groups. We say $\varphi$ is \emph{strictly semifinite} if its restriction to $M^\varphi$ is semifinite; that is, if $\dom(\varphi)\cap M^\varphi$ is dense in $M^\varphi$ in the strong (or weak) operator topology. This property has several equivalent characterizations (see \cite[Lemma 1.2]{GGLN25}), but the most relevant to this article is the existence of a faithful normal conditional expectation $\E_\varphi\colon M\to M^\varphi$. After \cite{Con72}, we say a faithful normal semifinite weight $\varphi$ is \emph{almost periodic} if its modular operator $\Delta_\varphi$ is diagonalizable. We will write $\Sd(\varphi)$ for the point spectrum of $\Delta_\varphi$, so that
    \[
        \Delta_\varphi = \sum_{\delta\in \Sd(\varphi)} \delta 1_{\{\delta\}}(\Delta_\varphi)
    \]
whenever $\varphi$ is almost periodic. Note that an almost periodic weight is automatically strictly semifinite by \cite[Proposition 1.1]{Con74}.

\section{Almost Unimodular Groups}\label{sec:almost_unimodular_groups}

\begin{thm}[{Theorem~\ref{introthm:A}}]\label{thm:when_is_plancherel_weight_almost_periodic}
Let $G$ be a locally compact group equipped with a left Haar measure $\mu_G$, and let $\vphi_G$ be the associated Plancherel weight on $L(G)$. Denote the modular function by $\Delta_G\colon G\to \R_+$. The following are equivalent:
    \begin{enumerate}[label=(\roman*)]
        \item $\ker{\Delta_G}$ is open; \label{part:open_unimodular_part}
        
        \item $\varphi_G$ is strictly semifinite; \label{part:strictly_semifinite}
        
        \item $\varphi_G$ is almost periodic;\label{part:almost_periodic}
        
        \item $\Delta_G$ viewed as an operator affiliated with $L^\infty(G) \subset B(L^2(G))$ has non-empty point spectrum. \label{part:mod_functiona_diagonlizable}
    \end{enumerate}
In this case one has
    \[
        \Sd(\varphi_G)=\Delta_G(G) \qquad \text{ and } \qquad L(G)^{\vphi_G} = \{\lambda_G(s)\colon s\in \ker{\Delta_G}\}''\cong L(\ker{\Delta_G}).
    \]
Under the identification $L(G)^{\varphi_G} \cong L(\ker{\Delta_G})$, $\varphi_G|_{L(G)^{\varphi_G}}$ is the Plancherel weight on $L(\ker{\Delta_G})$ corresponding to the left Haar measure $\mu_G|_{\mathcal{B}(\ker{\Delta_G})}$.
\end{thm}
\begin{proof}
Throughout we denote $G_1:= \ker{\Delta_G}$.\\

\noindent \textbf{\ref{part:open_unimodular_part}$\Rightarrow$\ref{part:strictly_semifinite}:} Using that $G_1$ is open, we can identify $C_c(G_1)$ as a subalgebra of $C_c(G)$. Under this identification, $f\in C_c(G_1)$ satisfies
    \[
        \sigma_t^{\varphi_G}( \lambda_G(f)) = \int_G \sigma_t^{\varphi_G}(\lambda_G(s)) f(s)\ d\mu_G(s) = \int_G \Delta_G(s)^{it} \lambda_G(s) f(s)\ d\mu_G(s) = \int_G \lambda_G(s) f(s)\ d\mu_G(s) = \lambda_G(f),
    \]
for all $t\in \R$, where we have used that $\supp(f)\subset G_1$. Thus $\lambda_G(f)\in L(G)^{\varphi_G}$. Additionally, the definition of the Plancherel weight gives
    \[
        \lambda_G( C_c(G_1)* C_c(G_1)) \subset \lambda_G( C_c(G)* C_c(G))\cap L(G)^{\varphi_G} \subset \dom(\varphi_G|_{L(G)^{\varphi_G}}).
    \]
Since $G_1$ is an open neighborhood of $e$ in $G$, $\lambda_G(C_c(G_1)*C_c(G_1))$ contains an approximate unit, and therefore $\varphi_G|_{L(G)^{\varphi_G}}$ is semifinite.\\

\noindent\textbf{\ref{part:strictly_semifinite}$\Rightarrow$\ref{part:almost_periodic}:} This follows from \cite[Lemma 1.4]{GGLN25} after noting that $\sigma_t^{\varphi_G}(\lambda_G(s)) = \Delta_G(s)^{it} \lambda_G(s)$ for all $s\in G$, but we also provide the following direct proof. Let $e_{\vphi_G}\in B(L^2(L(G),\vphi_G))$ be the projection onto the subspace $L^2(L(G)^{\vphi_G},\vphi_G)$. Since $\lambda_G(s)\in L(G)^{\vphi_G}$ if and only if $s\in G_1$, one has that $\{\lambda_G(s)e_{\vphi_G} \lambda_G(s)^*\colon sG_1\in G/G_1\}$ is a pairwise orthogonal family of projections in $\<L(G),e_{\vphi_G}\>$ (the basic construction for the inclusion $L(G)^{\vphi_G}\subset L(G)$). Moreover, $\lambda_G(s)e_{\vphi_G} \lambda_G(s)^*=\lambda_G(t)e_{\vphi_G} \lambda_G(t)^*$ whenever $sG_1=tG_1$. This implies the projection 
    \[
        \sum_{sG_1\in G/G_1} \lambda_G(s)e_{\vphi_G} \lambda_G(s)^*
    \]
is a central projection in $\<L(G),e_{\vphi_G}\>$ that dominates $e_{\vphi_G}$. The strict semifiniteness assumption implies $e_{\vphi_G}$ has full central support (see \cite[Proposition 2.3.(b)]{GGLN25}), and therefore the above projection must be $1$. Since $\lambda_G(s)e_{\vphi_G} \lambda_G(s)^*$ is the projection onto to $\Delta_G(s)$-eigenspace of $\Delta_{\vphi_G}$, it follows that $\Delta_{\vphi_G}$ is diagonalizable with
    \begin{align}
        \Delta_{\vphi_G} = \sum_{sG_1\in G/G_1} \Delta_G(s) \lambda_G(s)e_{\vphi_G} \lambda_G(s)^*. \label{eqn:diagonalization_mod_op}
    \end{align}
That is, $\vphi_G$ is almost periodic.\\

\noindent \textbf{\ref{part:almost_periodic}$\Rightarrow$\ref{part:mod_functiona_diagonlizable}:} This follows from the fact that $C_c(G)\ni f\mapsto \lambda_G(f)$ extends to a unitary $L^2(G)\to L^2(L(G),\varphi_G)$ which carries $\Delta_G$ to $\Delta_{\varphi_G}$.\\

\noindent \textbf{\ref{part:mod_functiona_diagonlizable}$\Rightarrow$\ref{part:open_unimodular_part}:} The continuity of $\Delta_G$ as a function implies that its point spectrum as an operator affiliated with $L^\infty(G)$ is the set $\{\delta\in \Delta_G(G) \colon \mu_G(\Delta_G^{-1}(\{\delta\}))>0\}$. Thus if we let $f\in L^2(G)$ be an eigenvector of $\Delta_G$, then its eigenvalue $\delta$ is necessarily in $\Delta_G(G)$. Additionally, for $s\in \Delta_G^{-1}(\{1/\delta\})$ we have
    \[
        \Delta_G \lambda_G(s) f = \frac{1}{\delta} \lambda_G(s) \Delta_G f = \lambda_G(s) f.
    \]
Thus $1$ is in the point spectrum of $\Delta_G$, and by definition of $\Delta_G$ as a pointwise multiplication operator we must have
    \[
        \mu_G\left(\{t\in G\setminus G_1\colon |[\lambda_G(s) f](t)|>0\}\right) =0.
    \]
Consequently, $1_{G_1} \lambda_G(s) f$ is a non-trivial square integrable function supported in $G_1$. It must therefore be the case that there exists a Borel set $E\subset G_1$ with $0<\mu_G(E) < \infty$, and so $G_1$ is open by Lemma~\ref{lem:open_via_Haar_measure}.\\

\noindent We now verify the final claim. The equality $\Sd(\varphi_G)=\Delta_G(G)$ follows from (\ref{eqn:diagonalization_mod_op}), and Lemma~\ref{lem:open_via_Haar_measure} gives that $\mu_G|_{\B(G_1)}$ is a left Haar measure on $G_1$. Next, from
    \[
        \sigma_t^{\varphi_G}(\lambda_G(s)) = \Delta_G(s)^{it} \lambda_G(s) \qquad \qquad t\in \R,\ s\in G.
    \]
we immediately have the inclusion $\{\lambda_G(s)\colon s\in G_1\}''\subset L(G)^{\vphi_G}$. Conversely, let $\mathcal{E}_{\vphi_G}\colon L(G)\to L(G)^{\vphi_G}$ be the unique $\vphi_G$-preserving conditional expectation, which exists by the strict semifiniteness of $\varphi_G$. The above formula implies $\mathcal{E}_{\vphi_G}(\lambda_G(s)) = 1_{G_1}(s) \lambda_G(s)$, and using the normality of $\mathcal{E}_{\vphi_G}$ we have for $f\in C_c(G)$ that
    \[
        \mathcal{E}_{\vphi_G}(\lambda_G(f)) = \int_G \mathcal{E}_{\vphi_G}(\lambda_G(s)) f(s) d\mu_G(s) = \int_{G} \lambda_G(s) 1_{G_1}(s) f(s)\ d\mu(s)  \in \{\lambda_G(s)\colon s\in G_1\}''.
    \]
Since we can approximate elements of $L(G)^{\vphi_G}$ by integrals of the above form, it follows that $\{\lambda_G(s)\colon s\in G_1\}''=L(G)^{\varphi_G}$. Finally, the Plancherel weight $\varphi_1$ associated to $\mu_G|_{\B(G_1)}$ is determined by the full left Hilbert algebra generated by $C_c(G_1)\subset L^2(G_1, \mu_G)$ (see \cite[Section VII.3]{Tak03}). Since we can identify these spaces as subspaces of $C_c(G)$ and $L^2(G)$, respectively, it follows that $\varphi_1$ is the restriction of $\varphi_G$ to $\{\lambda_G(s)\colon s\in G_1\}''=L(G)^{\varphi_G}$.
\end{proof}

In light of the previous theorem and the discussion in the introduction, we make the following definition.

\begin{defi}\label{defi:almost_unimodular_group}
Let $G$ be a locally compact group with modular function $\Delta_G\colon G\to \R_+$. We say $G$ is \textbf{almost unimodular} if $\ker{\Delta_G}$ is open in $G$. $\hfill\blacksquare$    
\end{defi}

It follows from Lemma~\ref{lem:dichotomy} that $G$ is almost unimodular if and only if the map from (\ref{eqn:quotient_map}) fails to be a homeomorphism onto $\R_+$. Thus any locally compact group $G$ with $\Delta_G(G)\subsetneq \R_+$ is almost unimodular. This includes all unimodular groups, and, in fact, we will see below that the class of almost unimodular groups is the smallest class of locally compact groups that contains all unimodular groups and is closed under extensions by discrete groups (see Corollary~\ref{cor:class_of_almost_unimod_groups}). Before we present more concrete examples, we deduce some alternate characterizations of almost unimodularity.

\begin{prop}\label{prop:plancherel_weight_almost_periodic_sigma_compact_and_second_countable_case}
Let $G$ be a locally compact group equipped with a left Haar measure $\mu_G$, and let $\Delta_G\colon G\to \R_+$ be the modular function.
    \begin{enumerate}[label=(\alph*)]
        \item $G$ is almost unimodular if and only if $\ker{\Delta_G}$ contains the connected component of $e\in G$. \label{part:connected_component}
        
        \item If $G$ is $\sigma$-compact, then $G$ is almost unimodular if and only if $\mu_G(\ker{\Delta_G})>0$. \label{part:sigma_compact}

        \item If $G$ is second countable, then $G$ is almost unimodular if and only if $\Delta_G(G)$ is countable. \label{part:second_countable}
    \end{enumerate}
\end{prop}
\begin{proof} \textbf{(a):} If $G$ is almost unimodular then $\ker{\Delta_G}$ contains the connected component of $e$ because it is an open and closed neighborhood of $e$. Conversely, if $\ker{\Delta_G}$ contains the connected component of $e$ then $G/\ker{\Delta_G}$ is totally disconnected. This quotient therefore cannot be isomorphic to $\R_+$ and hence Lemma~\ref{lem:dichotomy} implies $\ker{\Delta_G}$ is open.\\

\noindent\textbf{(b):} If $G$ is almost unimodular, then $\ker{\Delta_G}$ has positive Haar measure as an open set. Conversely, suppose $\mu_G(\ker{\Delta_G})>0$. The $\sigma$-compactness of $G$ implies $\mu_G$ is $\sigma$-finite, and consequently there exists a Borel subset $E\subset \ker{\Delta_G}$ satisfying $0<\mu_G(E)<\infty$. Thus $\ker{\Delta_G}$ is open by Lemma~\ref{lem:open_via_Haar_measure}.\\

\noindent\textbf{(c):} The second countability of $G$ implies $L^2(G) \cong L^2(L(G),\varphi_G)$ is separable as a Hilbert space, and consequently the point spectrum of $\Delta_{\varphi_G}$ is necessarily countable. If $G$ is almost unimodular, then the point spectrum of $\Delta_{\varphi_G}$ is $\Sd(\varphi_G)= \Delta_G(G)$ by Theorem~\ref{thm:when_is_plancherel_weight_almost_periodic}. Conversely, if $\Delta_G(G)$ is countable, then
    \[
        G = \bigsqcup_{\delta \in \Delta_G(G)} \Delta_G^{-1}(\{\delta\})
    \]
and
    \[
        \mu_G(\ker{\Delta_G}) = \mu_G(s\cdot \ker{\Delta_G}) = \mu_G\left(\Delta_G^{-1}(\{\Delta_G(s)\})\right)
    \]
imply $\mu_G(\ker{\Delta_G})>0$. Hence $G$ is almost unimodular by part \ref{part:sigma_compact}.
\end{proof}

\begin{exs}\label{exs:first_examples}
    \begin{enumerate}
    \item[]

    \item\label{ex:ax+b_group} Let $\R_+\curvearrowright \R$ be the action by multiplication. If we give $\R_+$ and $\R$ their usual topologies, then $G:=\R_+\ltimes \R$ is \emph{not} almost unimodular since
        \[
            \ker{\Delta_G} = \{1\}\times \R
        \]
    is not open. %The same is true when $\R$ is equipped with the discrete topology, but in this case $\mu_G(\{1\}\times \R)=\infty$ by outer regularity, and so Proposition~\ref{prop:plancherel_weight_almost_periodic_sigma_compact_and_second_countable_case}.\ref{part:sigma_compact} is not true in general. 
    If we instead give $\R_+$ the discrete topology, then the above kernel is open so that $G$ is almost unimodular by definition. This demonstrates that Proposition~\ref{prop:plancherel_weight_almost_periodic_sigma_compact_and_second_countable_case}.\ref{part:second_countable} fails without the assumption of second countability since $\Delta_G(G)=\R_+$.

    \item Let $\R_+\curvearrowright \R^2$ be the diagonal multiplication action: $a\cdot (s,t) = (as, at)$. Endow $\R_+$ with its usual topology and the first and second copies of the real numbers in $\R\times \R=\R^2$ with the usual and discrete topologies, respectively. As in the previous example, $G:= \R_+ \ltimes \R^2$ is \emph{not} almost unimodular because
        \[
            \ker{\Delta_G} = \{1\}\times \R^2
        \]
    is not open in $G$. However, it follows from outer regularity that $\mu_G(\ker{\Delta_G})=+\infty$, and so Proposition~\ref{prop:plancherel_weight_almost_periodic_sigma_compact_and_second_countable_case}.\ref{part:sigma_compact} fails without the assumption of $\sigma$-compactness.
    
    \item\label{ex:p-adic_ax+b_group} For a prime number $p\in \N$, let $\Q_p$ be the p-adic rationals and $\Q_p^\times = \Q_p\setminus \{0\}$, which are second countable locally compact groups under addition and multiplication, respectively. Both groups are abelian and hence unimodular. Let $\Q_p^\times \curvearrowright \Q_p$ be the multiplication action, which scales the left Haar measure of $\Q_p$ by powers of $p$. Then $G:=\Q_p^\times \ltimes \Q_p$ is second countable and satisfies
        \[
            \Delta_{G}(G) = p^\Z \qquad \text{ and }\qquad \ker{\Delta_G} = \Z_p^\times \ltimes \Q_p,
        \]
    (see \cite[Section 15.29]{HR79} or \cite[Subsection 1.2]{KT13}). Hence $G$ is almost unimodular either by definition or by Proposition~\ref{prop:plancherel_weight_almost_periodic_sigma_compact_and_second_countable_case}.\ref{part:second_countable}.

    \item\label{ex:totally_disconnected} A totally disconnected group $G$ is almost unimodular by Proposition~\ref{prop:plancherel_weight_almost_periodic_sigma_compact_and_second_countable_case}.\ref{part:connected_component}. Recall from van Dantzig's Theorem \cite{vD36} (see also \cite[Theorem 3.1]{Wil04}), $G$ admits a neighborhood basis for the identity consisting of compact open subgroups. For any such compact open subgroup $U\leq G$ and $t\in G$ one has
        \begin{align*}
            [U : U\cap tUt^{-1}] &= [ t^{-1} Ut : t^{-1} U t \cap U ] = \frac{\mu_G(t^{-1} Ut)}{\mu_G(t^{-1} U t \cap U)}\\
                &= \frac{\mu_G(U) \Delta_G(t)}{\mu_G(t^{-1} U t \cap U)} = [U : t^{-1} U t \cap U] \Delta_G(t).
        \end{align*}
    The above quantities are finite since $U$ is compact and $U\cap t U t^{-1}$ is open, and so it follows that $\Delta_G(G)\leq \Q_+$. In fact, $\Delta_G(t) = \frac{s(t)}{s(t^{-1})}$, where $s\colon G\to \N$ is the \emph{scale function} (see \cite[Proposition 4.1]{Wil04}). Note that $\ker{\Delta_G}$  is also totally disconnected and consists of the $t\in G$ such that $s(t)=s(t^{-1})$.

    \item\label{ex:graph_automorphisms} Let $\Gamma$ be a connected locally finite (simple) graph $\Gamma$. Then $\Aut(\Gamma)$ equipped with the compact-open topology (i.e. the topology of pointwise convergence on $\Gamma$) is a totally disconnected group (see \cite[Example 2.1.(c)]{Wil04}), and hence are almost unimodular by the previous example. Moreover, it follows from \cite[Lemma 1.(iii)]{Sch79} or \cite[Theorem 1]{Tro85} that $\Aut(\Gamma)$ is unimodular if and only if
        \[
            |\Aut(\Gamma)_w\cdot v| = |\Aut(\Gamma)_v\cdot w|,
        \]
    for all vertices $v,w$ in the same orbit under $G$, where $\Aut(\Gamma)_w\leq \Aut(\Gamma)$ is the stabilizer of $w$. For an example of a graph $\Gamma$ with non-unimodular $\Aut(\Gamma)$, see the discussion at the end of \cite{Tro85}.
    $\hfill\blacksquare$
    \end{enumerate}
\end{exs}

\begin{rem}
Theorem~\ref{thm:when_is_plancherel_weight_almost_periodic} implies that a necessary condition for $G$ to be almost unimodular is that the restriction of $\varphi_G$ to $\{\lambda_G(s)\colon s\in \ker{\Delta_G}\}''$ corresponds to a Plancherel weight $\varphi_1$ under the identification $\{\lambda_G(s)\colon s\in \ker{\Delta_G}\}''\cong L(\ker{\Delta_G})$. This is, in fact, also a sufficient condition. Indeed, for a non-zero $f\in C_c(\ker{\Delta_G})$ we have
    \[
        \int_G |f|^2\ d\mu_G =\varphi_G(\lambda_G(f)^* \lambda_G(f)) = \varphi_1(\lambda_{\ker{\Delta_G}}(f)^* \lambda_{\ker{\Delta_G}}(f)) = \int_{\ker{\Delta_G}} |f|^2\ d\mu_{\ker{\Delta_G}} \in (0,\infty),
    \]
for some left Haar measure $\mu_{\ker{\Delta_G}}$ on $\ker{\Delta_G}$. Consequently, $\supp(f)\subset \ker{\Delta_G}$ must satisfy $0 < \mu_G(\supp(f)) < \infty$, and therefore $\ker{\Delta_G}$ is open by Lemma~\ref{lem:open_via_Haar_measure}.$\hfill\blacksquare$
\end{rem}

\begin{rem}
As noted after Definition~\ref{defi:almost_unimodular_group}, a locally compact group $G$ is \emph{not} almost unimodular if and only if the map from (\ref{eqn:quotient_map}) is an isomorphism onto $\R_+$. It follows from Theorem~\ref{thm:when_is_plancherel_weight_almost_periodic} that these groups can be characterized by $\Delta_G$ having empty point spectrum or by $\sqrt{\dom}(\varphi_G)\cap L(\ker{\Delta_G}) = \{0\}$. Among second countable groups, these are precisely the $G$ with $\Delta_G(G)=\R_+$ by Proposition~\ref{prop:plancherel_weight_almost_periodic_sigma_compact_and_second_countable_case}.\ref{part:second_countable}.$\hfill\blacksquare$
\end{rem}

\section{Permanence Properties}\label{sec:permanence_properties}

Let $G$ be a locally compact group with closed subgroup $H\leq G$. Recall that a \emph{rho-function} for the pair $(G,H)$ is a continuous function $\rho\colon G\to (0,\infty)$ satisfying
    \[
        \rho(st) = \frac{\Delta_H(t)}{\Delta_G(t)} \rho(s) \qquad \qquad s\in G,\ t\in H.
    \]
Such functions always exist (see, for example, \cite[Proposition 2.56]{Fol16}).

\begin{prop}\label{prop:almost_unimodular_subgroups}
Let $G$ be an almost unimodular group with closed subgroup $H\leq G$, and let $\rho\colon G\to (0,\infty)$ be a rho-function for the pair $(G,H)$. Then $H$ is almost unimodular if and only if $\{t\in H\colon \rho(t)=\rho(e)\}$ is open in $H$. In particular, all closed normal subgroups of $G$ are almost unimodular.
\end{prop}
\begin{proof}
Observe that
    \[
        \frac{\rho(t)}{\rho(e)} = \frac{\Delta_H(t)}{\Delta_G(t)} \qquad \qquad t\in H, 
    \]
and consequently $H\ni t\mapsto \frac{\rho(t)}{\rho(e)}$ is a continuous homomorphism. It follows that $R:=\{t\in H\colon \rho(t)=\rho(e)\}$ is a non-empty closed subgroup of $H$, satisfying
    \[
        \ker{\Delta_H}\supset R\cap \ker{\Delta_G} \qquad \text{ and } \qquad R \supset \ker{\Delta_H} \cap \ker{\Delta_G}.
    \]
Thus if $H$ is almost unimodular, then $\ker{\Delta_H} \cap \ker{\Delta_G}$ is an open subgroup of $R$ and therefore $R$ is open. Similarly, when $R$ is assumed to be open it follows that $\ker{\Delta_H}$ is open.

For the last statement, the normality of $H$ implies that $G/H$ admits a left Haar measure. It then follows from \cite[Theorem 2.49]{Fol16} that $\Delta_G|_H = \Delta_H$, so that $R=H$ is open and therefore $H$ is almost unimodular by the above.
\end{proof}

Let $G$ be a locally compact group with closed normal subgroup $N\unlhd G$. Since $G/N$ is a locally compact group, we fix a left Haar measure $\mu_{G/N}$ on $G/N$ such that 
    \[
        \int_G f(s) d\mu_G(s) =  \int_{G/N} \int_N f(sn) d\mu_N(n) d\mu_{G/N}(sN) \qquad \qquad f \in L^1(G),
    \]  
where $\mu_G$ and $\mu_N$ are fixed left Haar measures on $G$ and $N$, respectively (see \cite[Theorem 2.49]{Fol16}). Since for each $s \in G$, $\mu_N \circ \Ad{s}$ is a left Haar measure on $N$ and $(s,n) \to sns^{-1}$ is a continuous homomorphism, we have that the Radon-Nikodym derivative $\frac{d \mu_N\circ \Ad{s}}{d \mu_N}$ is a constant and the map $s \to \frac{d \mu_N\circ \Ad{s}}{d \mu_N}$ is a continuous homomorphism. Thus one can compute the modular function of $G/N$ to obtain
\begin{align}\label{eqn:mod_function_of_quotient_group}
        \Delta_{G/N}(sN) = \Delta_G(s) \frac{d \mu_N\circ \Ad{s}}{d \mu_N} \qquad\qquad s\in G.
    \end{align}
This can be found in \cite[Lemma 3.4]{Tat72} and the discussion preceding it. The expression is different since we have adopted the convention of using left Haar measures, but the proof is the same.

\begin{prop}\label{prop:quotient_groups}
Let $G$ be a locally compact group with closed normal subgroup $N \unlhd G$ and define a continuous homomorphism $\phi\colon G\to (0,\infty)$ by 
    \[
        \phi(s):= \frac{d \mu_N\circ \Ad{s}}{d \mu_N} \qquad \qquad s\in G.
    \]
Then any two of the following statements below imply the third:
    \begin{enumerate}[label=(\roman*)]
        \item $G$ is almost unimodular;
        \item $G/N$ is almost unimodular;
        \item $\ker{\phi}$ is open.
    \end{enumerate}
In particular, if $N$ is compact or discrete then $G$ is almost unimodular if and only if $G/N$ is almost unimodular.
\end{prop}
\begin{proof}
Letting $q\colon G\to G/N$ be the quotient map, observe
    \[
        \ker{\Delta_{G/N}\circ q} = q^{-1}(\ker{\Delta_{G/N}}),
    \]
is open in $G$ if and only if $\ker{\Delta_{G/N}}$ is open in the quotient topology. Then using (\ref{eqn:mod_function_of_quotient_group}) we have
    \begin{align*}
        \ker{\Delta_G} &\supset \ker{\phi}\cap \ker{\Delta_{G/N}\circ q}\\
        \ker{\Delta_{G/N} \circ q} & \supset \ker{\phi}\cap \ker{\Delta_G}\\
        \ker{\phi} & \supset \ker{\Delta_G}\cap \ker{\Delta_{G/N}\circ q}.
    \end{align*}
Thus the openness of any two of $\ker{\Delta_G}$, $\ker{\Delta_{G/N}\circ q}$, or $\ker{\phi}$ implies openness of the third.

When $N$ is compact (resp. discrete) we can take $\mu_N$ to be the unique left Haar measure with $\mu_N(N)=1$ (resp. $\mu_N(\{e\})=1$). This uniqueness implies $\phi\equiv 1$, and hence $\ker{\phi}=G$ is open. Therefore the previous part implies that $G$ is almost unimodular if and only if $G/N$ is almost unimodular.
\end{proof}

When the quotient group $G/N$ is unimodular, (\ref{eqn:mod_function_of_quotient_group}) implies $\ker{\Delta_G} = \ker{\phi}$ and the previous proposition is a tautology. Nevertheless, there are sufficient conditions in this case to guarantee that $G$ is almost unimodular:

\begin{prop}\label{prop:short_exact_sequences}
Let
    \[
        1 \to N \to G \to H\to 1
    \]
be a short exact sequence of locally compact groups (that is, $N$ is a closed normal subgroup and $H\cong G/N$). If $N$ is almost unimodular and $H$ is discrete, then $G$ is almost unimodular.
\end{prop}
\begin{proof}
The discreteness of $H$ implies that $N$ is open in $G$. Consequently, $\ker{\Delta_N}$ is open in $G$, and as a subgroup of $\ker{\Delta_G}$ it follows that $G$ is almost unimodular.    
\end{proof}

By definition, any almost unimodular group $G$ appears in a short exact sequence of the form in previous proposition:
    \begin{align}\label{eqn:almost_unimodular_definition_SES}
        1 \to \ker{\Delta_G} \to G \to \Delta_G(G) \to 1.
    \end{align}
Thus we immediately have the following corollary:

\begin{cor}\label{cor:class_of_almost_unimod_groups}
The class of almost unimodular groups is the smallest class of locally compact groups that is closed under extensions by discrete groups and that contains all unimodular groups.
\end{cor}

Proposition~\ref{prop:short_exact_sequences} can be refined in the special case of (continuous) cocycle semidirect products. Recall that for groups $H$ and $N$, a \emph{cocycle action} $(\alpha,c)\colon H \curvearrowright N$ is a pair of maps $\alpha : H \to \Aut(N)$ and $c: H \times H \to N$ satisfying the relations
    \begin{align}\label{eqn:relations_of_cocycle_type}
        \alpha_s\alpha_t = \Ad{c(s,t)} \alpha_{st} \qquad \text{ and } \qquad c(s,t)c(st,r) = \alpha_s(c(t,r)) c(s,tr) \qquad s,t,r\in H.
    \end{align}
One can (and we will) always normalize the $2$-cocycle so that $c(s,e)=c(e,s)=e$ for all $s\in H$. For locally compact groups $H$ and $N$, we say the cocycle action $(\alpha,c)$ is \emph{continuous} if the maps $H\times N\ni (s,x)\mapsto \alpha_s(x)\in N$ and $H\times H\ni (s,t)\mapsto c(s,t)\in N$ are continuous (where the products are equipped with the product topology). In this case, the \emph{cocycle semidirect product} of this action is a locally compact group, denoted by $H \sltimes{(\alpha,c)} N$, consisting of the set $H \times N$ equipped with the product topology and the following group operations 
\[
    (s,x) (t,y) = (st, c(t^{-1},s^{-1})^{-1} \alpha_{t^{-1}}(x)y) \qquad \text{ and } \qquad (s,x)^{-1} = (s^{-1}, \alpha_s(x)^{-1}c(s,s^{-1})),
\]
where $(s,x),(t,y) \in H \sltimes{(\alpha,c)} N$. Note that a left Haar measure for this group is given by the Radon product $\mu_H \hat{\times} \mu_N$ for any left Haar measures $\mu_H$ and $\mu_N$ of $H$ and $N$, respectively. Additionally, using the relations in (\ref{eqn:relations_of_cocycle_type}), one can show
    \[
        \frac{d(\mu_N \circ \alpha_{st})}{d\mu_N}= \Delta_N(c(s,t)) \frac{d(\mu_N \circ \alpha_{s})}{d\mu_N}\frac{d(\mu_N \circ \alpha_{t})}{d\mu_N} \qquad \qquad s,t \in H.
    \]
It follows that the modular function for $H \sltimes{(\alpha,c)} N$ is given by
\begin{align}\label{eqn:cocycle_semidirect_product_modular_function}
   \Delta_{H \sltimes{(\alpha,c)} N}(s,x) : = \Delta_H(s) \Delta_N(x) \left(\Delta_N\left(c(s,s^{-1})\right) \frac{d(\mu_N \circ\alpha_{s})}{d\mu_N}\right)^{-1} \qquad \qquad (s,x) \in H \sltimes{(\alpha,c)} N.
\end{align}
Note that if $c$ is valued in $\ker{\Delta_N}$, then the right-hand side is also a homomorphism on $H\times N$.

\begin{prop}\label{prop:cocycle_semidirect_products}
Let $(\alpha,c)\colon H \curvearrowright N$ be a continuous cocycle action of almost unimodular groups. Then $H \sltimes{(\alpha,c)} N$ is almost unimodular if and only if 
\[
    \left\{ s \in H : \Delta_N\left(c(s,s^{-1})\right)\frac{d(\mu_N \circ\alpha_{s})}{d \mu_N} = 1\right\} 
\]   
is open in $H$. In particular, if $H$ is discrete then $H \sltimes{(\alpha,c)} N$ is always almost unimodular. 
\end{prop}
\begin{proof}
Identifying $N \cong \{(e,x)\colon x\in N\}\leq H \sltimes{(\alpha,c)} N$, it follows that $s\mapsto (s,e)N$ defines an isomorphism from $H$ to $H \sltimes{(\alpha,c)} N/ N$. So comparing (\ref{eqn:cocycle_semidirect_product_modular_function}) with (\ref{eqn:mod_function_of_quotient_group}), we see that in this case the continuous homomorphism $\phi\colon H\sltimes{(\alpha,c)} N\to (0,\infty)$ has the form
    \[
        \phi(s,x) = \Delta_N(x)^{-1} \Delta_N\left(c(s,s^{-1})\right) \frac{d(\mu_N \circ\alpha_{s})}{\mu_N} \qquad (s,x)\in H\sltimes{(\alpha,c)} N.
    \]
Since $H \sltimes{(\alpha,c)} N/ N\cong H$ is assumed to be almost unimodular, Proposition~\ref{prop:quotient_groups} implies the cocycle semidirect product is almost unimodular if and only of $\ker{\phi}$ is open. Using that $\phi(s,x)= \Delta_N(x)^{-1} \phi(s,e)$, we see that
    \[
        \ker{\phi} = \bigsqcup_{\delta\in \Delta_N(N)} \{s\in H\colon \phi(s,e) = \delta^{-1}\} \times \Delta_N^{-1}(\{\delta\})
    \]
and consequently
    \[
        \ker{\phi}\cap \{(e,x)\colon x\in \ker{\Delta_N}\} = \{s\in H\colon \phi(s,e) = 1\}\times \ker{\Delta_N}.
    \]
Thus if $\{s\in H\colon \phi(s,e) = 1\}$ is open in $H$ then the right-hand side above is an open neighborhood of $(e,e)$ contained in the subgroup $\ker{\phi}$, which is therefore open. Conversely, if $\ker{\phi}$ is open, then the left-hand side above is open by the almost unimodularity of $N$, and consequently the projection onto its first coordinate, $\{s\in H\colon \phi(s,e) = 1\}$, is open in $H$.
\end{proof}

The following example, which will be needed in Section~\ref{sec:group_vNas}, shows that almost unimodular groups can always be realized as a continuous cocycle action of the image of the modular function and its kernel.

\begin{ex}\label{ex:cocycle_realization_of_almost_unimodular_groups}
Let $G$ be an almost unimodular group. Using the short exact sequence in (\ref{eqn:almost_unimodular_definition_SES}), choose a normalized section $\sigma: \Delta_G(G) \to G$ for $\Delta_G\colon G\to \Delta_G(G)$; that is, $\sigma(1)=e$ and $\Delta_G\circ \sigma(\delta)=\delta$ for all $\delta\in \Delta_G(G)$. Define maps
    \begin{align*}
        \alpha\colon \Delta_G(G) &\to \Aut(\ker{\Delta_G}) \qquad\qquad \text{ and }  &c \colon \Delta_G(G) \times \Delta_G(G) &\to \ker{\Delta_G}\\
            \delta &\mapsto \Ad{\sigma(\delta)} &(\delta_1,\delta_2) &\mapsto \sigma(\delta_1) \sigma(\delta_2)\sigma(\delta_1\delta_2)^{-1}.
    \end{align*}    
Then $(\alpha,c)\colon \Delta_G(G) \curvearrowright \ker{\Delta_G}$ is a continuous cocycle action and we can identify the cocycle semidirect product $\Delta_G(G)\sltimes{(\alpha,c)} \ker{\Delta_G}$ with $G$ via $(\delta, s) \mapsto s\sigma(\delta)$. 
$\hfill \blacksquare$
\end{ex}

\begin{rem}\label{rem:Galois_correspondence}
Any intermediate closed subgroup $\ker{\Delta_G}\leq H\leq G$ is automatically normal in $G$ since $G/\ker{\Delta_G}\cong \Delta_G(G)$ is abelian. Additionally, there is a 1-1 correspondence between such subgroups and subgroups $\Gamma\leq \Delta_G(G)$ given by $H\mapsto \Delta_G(H)$. Indeed, the inverse of this map is given by $\Gamma\mapsto \Gamma \sltimes{(\alpha,c)} \ker{\Delta_G}$, where $G\cong \Delta_G(G) \sltimes{(\alpha,c)} \ker{\Delta_G}$ is the identification from the previous example. This can be seen by noting that the containment $\ker{\Delta_G}\subset H$ implies $\Delta_G^{-1}(\Delta_G(H))=H$, whereas $\Delta_G\left(\Gamma \sltimes{(\alpha,c)} \ker{\Delta_G}\right)=\Gamma$ follows from $\sigma$ being a section associated to the modular function.$\hfill\blacksquare$
\end{rem}

Almost unimodularity is also preserved under fiber products:

\begin{prop}\label{prop:fiber_products}
Let $Q$ be a locally compact group and let $G_1,G_2$ be almost unimodular groups. Suppose $\theta_j\colon G_j\to Q$, $j=1,2$, are open, continuous homomorphisms. If $Q$ is almost unimodular, then $\{(s_1,s_2)\in G_1\times G_2\colon \theta_1(s_1)=\theta_2(s_2)\}$ is an almost unimodular group. If $\theta_1,\theta_2$ are surjective, then the converse holds.
\end{prop}
\begin{proof}
Denote $H:=\{(s_1,s_2)\in G_1\times G_2\colon \theta_1(s_1)=\theta_2(s_2)\}$, which forms a closed subgroup of $G_1\times G_2$. Consider the map
    \begin{align*}
        \chi\colon G_1\times G_2 &\to Q\\
            (s_1,s_2) &\mapsto \theta_1(s_1)\theta_2(s_2^{-1}), 
    \end{align*}
which is open and continuous by the assumptions on $\theta_1, \theta_2$ (though it need not be a homomorphism). Thus $X:=\chi(G_1\times G_2)\subset Q$ is an open subset, and the left Haar measure $\mu_Q$ of $Q$ restricts to a non-trivial Radon measure on $X$ that we denote by $\mu$. A direct computation shows that
    \begin{align*}
        \bar{\chi}\colon (G_1\times G_2)/H & \to X\\
         (s_1,s_2)H &\mapsto \chi(s_1,s_2)
    \end{align*}
is a well-defined bijection. Equipping $(G_1\times G_2)/H$ with the quotient topology, it is straightforward to verify that $\bar{\chi}$ is a homeomorphism, and so $\nu:=\mu\circ \bar{\chi}$ defines a non-trivial Radon measure on the homogeneous space $(G_1\times G_2)/H$. Observe that for a Borel set $E\subset (G_1\times G_2)/H$ and $(s_1,s_2)\in G_1\times G_2$ we have
    \[
        \nu( (s_1,s_2)\cdot E) = \mu( \theta_1(s_1) \bar{\chi}(E) \theta_2(s_2^{-1})) = \Delta_Q(\theta_2(s_2^{-1})) \nu( E).
    \]
Hence $\nu$ is a strongly quasi-invariant measure (see \cite[Section 2.6]{Fol16}), and \cite[Theorems 2.56 and 2.59]{Fol16} imply there is a rho-function $\rho$ for $(G_1\times G_2,H)$ satisfying
    \[
        \frac{\rho(s_1t_1,s_2t_2)}{\rho(t_1,t_2)} = \frac{d\nu((s_1,s_2)\ \cdot\ )}{d\nu}\left( (t_1,t_2)H\right) = \Delta_Q(\theta_2(s_2^{-1})) \qquad s_1,t_1\in G_1,\, s_2,t_2\in G_2.
    \]
In particular, one has
    \begin{align}\label{eqn:rho-function_for_fiber_products}
        \left\{ (t_1,t_2)\in H\colon \rho(t_1,t_2) = \rho(e,e) \right\} = H\cap [ G_1\times \theta_2^{-1}(\ker{\Delta_Q})]
    \end{align}
(Note that the above set also equals $H\cap [\theta_1^{-1}(\ker{\Delta_Q})\times G_2]$ by definition of $H$). Thus if $Q$ is almost unimodular, then the above is a relatively open subset of $H$, and hence $H$ is almost unimodular by Proposition~\ref{prop:almost_unimodular_subgroups}.

Now suppose $\theta_1,\theta_2$ are surjective and that $H$ is almost unimodular. The latter implies the set in (\ref{eqn:rho-function_for_fiber_products}) is open in $H$ by Proposition~\ref{prop:almost_unimodular_subgroups}. We will show that the coordinate projections $\pi_j\colon G_1\times G_2\to G_j$ for $j=1,2$ restrict to open and surjective maps on $H$; this will complete the proof because $\ker{\Delta_Q}$ will therefore be the open image of the set in (\ref{eqn:rho-function_for_fiber_products}) under the map  $\theta_2\circ \pi_2$. The surjectivity of $\pi_1|_H$ follows from that of $\theta_2$: for any $s_1\in G_1$ we can find $s_2\in G_2$ with $\theta_2(s_2) = \theta_1(s_1)$, and therefore $(s_1,s_2)\in H$. Similarly, the surjectivity of $\pi_2|_H$ follows from that of $\theta_1$. Toward showing the openness of these restrictions, consider $(s_1,s_2) \in H\cap (U_1\times U_2)$ with $U_1\subset G_1$ and $U_2\subset G_2$ open. Then
    \[
        V_j:= U_j\cap \theta_j^{-1}[\theta_1(U_1)\cap \theta_2(U_2)]
    \]
is an open neighborhood of $s_j$ for each $j=1,2$, and one has
    \[
        \pi_j( H\cap (U_1\times U_2)) \supset \pi_j( H\cap (V_1\times V_2)) = V_j.
    \]
Thus $\pi_j(H\cap (U_1\times U_2))$ is a neighborhood of $s_j$, and it follows that $\pi_j|_H$ are open maps for each $j=1,2$.
\end{proof}

\begin{rem}
By Goursat's Lemma, it is well-known for discrete groups that subdirect products  arise as special cases of fiber products. Under suitable assumptions, this can also be seen for locally compact groups as follows. Given locally compact groups $G_1,G_2$, for each $j=1,2$ let $\pi_j\colon G_1\times G_2 \to G_j$ and $\iota_j\colon G_j\to G_1\times G_2$ be the coordinate projection and injection, respectively. If $H\leq G_1\times G_2$ is a closed subgroup such that $\pi_j|_H$ surjective for each $j=1,2$, then Goursat's Lemma implies the map
    \begin{align*}
       \phi\colon G_1/ \iota_1^{-1}(H) &\to G_2/\iota_2^{-1}(H)\\
         s_1 \iota_1^{-1}(H) & \mapsto s_2 \iota_2^{-1}(H)
    \end{align*}
is a group isomorphism, where $s_2\in G_2$ is any element satisfying $(s_1,s_2)\in H$. If one further assumes each $\pi_j|_H$ is open for $j=1,2$, then this map becomes a homeomorphism and therefore defines an isomorphism of locally compact groups. Letting $Q$ be a representative of this isomorphism class, we obtain open, continuous surjections $\theta_j\colon G_j\to Q$ via the quotient maps, and one has 
    \[
        H= \{(s_1,s_2)\in G_1\times G_2\colon \theta_1(s_1) = \theta_2(s_2)\}.
    \]
Hence the subdirect product $H$ is also a fiber product. In particular, when $G_1$ and $G_2$ are almost unimodular, Proposition~\ref{prop:fiber_products} implies $Q$ is almost unimodular if and only if $H$ is almost unimodular.$\hfill\blacksquare$
\end{rem}

\begin{exs}
    \begin{enumerate}
    \item[]
    \item If $H,N$ are almost unimodular groups, then their direct product $H\times N$ is almost unimodular by Proposition~\ref{prop:cocycle_semidirect_products} and one has
        \[
            \Delta_{H\times N}(H\times N) = \Delta_H(H)\vee \Delta_N(N)
        \]
    and
        \[
            \ker{\Delta_{H\times N}} = \bigcup_{\delta\in \Delta_H(H)\cap \Delta_N(N)} \Delta_H^{-1}(\{\delta\})\times \Delta_N^{-1}(\{1/\delta\}).
        \]
    By considering any non-discrete almost unimodular group $H$, we see that $H \cong (H\times N)/N$ shows that the discreteness hypothesis in Proposition~\ref{prop:short_exact_sequences} is not necessary in general. This is in contrast to $N$ being normal in $H\times N$, which is necessary by Proposition~\ref{prop:almost_unimodular_subgroups}.

    \item\label{ex:restricted_direct_products}  Let $\{(G_i, K_i)\colon i\in I\}$ be a family of locally compact  groups with compact, open subgroups $K_i\leq G_i$. Their restricted direct product is the direct limit
    \[
        \prod_{i\in I} (G_i, K_i) := \lim_{F\to\infty} \left( \prod_{i\in F} G_i\right) \times \left(\prod_{i\in I\setminus F} K_i \right),
    \]
    over finite subsets $F\Subset I$, equipped with the final topology. Each infinite direct product is a locally compact group with left Haar measure given by the infinite product of left Haar measures $\mu_i$ on $G_i$ normalized so that $\mu_i(K_i)=1$. For each $i\in I$, $G_i$ is almost unimodular since $K_i$ is an open subset of $\ker{\Delta_{G_i}}$. The restricted direct product is itself almost unimodular since the kernel of the modular function of $\prod_{i\in I} (G_i, K_i)$ contains the open set $\prod_{i\in I} K_i$ and the image of its modular function is $\bigvee_{i\in I} \Delta_{G_i}(G_i) \leq \R_+$.
    
    \item Let $G,H$ be almost unimodular groups, and suppose $\{K_s \leq G\colon s\in H\}$ is a family of compact, open subgroups of $G$ such that $H\ni s\mapsto \frac{\mu_G(K_s)}{\mu_G(K_1)}$ is a continuous homomorphism (e.g. the trivial homomorphism if $K_s\equiv K_1$ for all $s\in H$). Then the translation action
        \[
            \alpha\colon H \curvearrowright \prod_{s\in H} (G,K_s)
        \]
    satisfies 
        \[
            \frac{d(\mu\circ \alpha_s)}{d\mu} = \frac{\mu_G(K_{s^{-1}})}{\mu_G(K_1)},
        \]
    where $\mu$ is the left Haar measure on the restricted direct product obtained from the family of left Haar measures $\{ \frac{1}{\mu_G(K_s)}\mu_G \colon s\in H\}$. Thus Proposition~\ref{prop:cocycle_semidirect_products} implies
        \[
            H \sltimes{\alpha} \prod_{s\in H} (G,K_s)
        \]
    is almost unimodular if and only if $\{s\in H\colon \mu_G(K_s) = \mu_G(K_1)\}$ is open. This holds, for example, when $K_s= K_1$ for all $s\in H$, in which case the above can be thought of as a kind of wreath product $(G,K_1) \wr H$. $\hfill\blacksquare$
    \end{enumerate}
\end{exs}

\begin{rem}
For a second countable locally compact group $G$ equipped with a continuous homomorphism $\omega\colon G\to (0,\infty)$, one has $\ker{\omega}$ is open if and only if $\omega(G)$ is countable. Indeed, $\ker{\omega}$ being open implies $\{\omega^{-1}(\{\delta\})\colon \delta\in \omega(G)\}$ is a disjoint family of open sets, and so second countability forces $\omega(G)$ to be countable. Conversely, if $\omega(G)$ is countable then we must have $\mu_G(\ker{\omega})>0$ since otherwise we obtain the contradiction
    \[
        \mu_G(G) = \sum_{\delta\in \omega(G)} \mu_G(\omega^{-1}(\{\delta\})) = \sum_{\delta\in \omega(G)} \mu_G(\ker{\omega}) = 0.
    \]
So $\mu_G(\ker{\omega})>0$, and as a subgroup of a second countable group it follows that $\ker{\omega}$ is open. Consequently, for second countable locally compact groups the open kernel statements in Propositions~\ref{prop:almost_unimodular_subgroups} and  \ref{prop:quotient_groups} can be replaced with countable image statements. One can also do this for Proposition~\ref{prop:cocycle_semidirect_products} when the $2$-cocycle $c\colon H\times H\to N$ is valued in $\ker{\Delta_N}$. $\hfill\blacksquare$
\end{rem}

\section{Square Integrable Representations}\label{sec:representation_theory}

In this section we establish connections between square integrable representations of an almost unimodular group $G$ and those of its unimodular part $\ker{\Delta_G}$. We begin, however, by characterizing when arbitrary representations of $G$ are induced by some representation of $\ker{\Delta_G}$. Fortunately, the discreteness of the quotient $\Delta_G(G)\cong G/\ker{\Delta_G}$ in this case greatly simplifies the description of an induced representation. For a representation $(\pi_1,\H_1)$ of $\ker{\Delta_G}$ and a fixed normalized section $\sigma\colon \Delta_G(G)\to G$ for $\Delta_G$, the induced representation $\Ind_{\ker{\Delta_G}}^G(\pi_1,\H_1)$ can be realized as
    \[
        \text{Ind}^G_{\ker{\Delta_G}} \H_1 \cong \ell^2(\Delta_G(G)) \otimes \H_1,
    \]
such that for all $f \in \ell^2(\Delta_G(G))\otimes \H_1,$
    \[
        [(\text{Ind}^G_{\ker{\Delta_G}}\pi_1)(s)f](\delta) = \pi_1\left(\sigma(\delta)^{-1} s\sigma ( \Delta_G(s)^{-1}\delta)\right) f(\Delta_G(s)^{-1}\delta), \qquad s\in G, \delta \in \Delta_G(G),
    \]
(see \cite[Proposition 2.3]{KT13}). Note that the above is independent of the choice of section $\sigma$, up to unitary equivalence.

\begin{thm}\label{thm:reps_induced_from_kernel}
Let $(\pi, \H)$ be a representation of an almost unimodular group $G$. The following are equivalent:
\begin{enumerate}[label=(\roman*)]
    \item $(\pi,\H)\cong \Ind^G_{\ker{\Delta_G}}(\pi_1,\H_1)$ for some representation $(\pi_1,\H_1)$ of $\ker{\Delta_G}$.

    \item There exists a representation $U: \Delta_G(G)\hat{\ } \to U(\H)$ satisfying $U_\gamma \pi(s) U_\gamma^* = (\gamma \mid \Delta_G(s)) \pi(s)$, where $(\cdot\mid \cdot)\colon \Delta_G(G)\hat{\ }\times \Delta_G(G)\to \mathbb{T}$ is the dual pairing.
    
\end{enumerate}
\end{thm}
\begin{proof} Throughout we denote $K := \Delta_G(G)\hat{\ }$, $G_1:= \ker\Delta_G$, and let $\sigma\colon \Delta_G(G)\to G$ be a fixed normalized section for $\Delta_G$.\\

\noindent \textbf{(i)$\Rightarrow$(ii):} We identify $\H$ with $\Ind^G_{G_1} \H_1$ to obtain $\H = \bigoplus_{\delta \in \Delta_G(G)} \pi(\sigma(\delta))\H_1$. Set $p_{\delta} := [\pi(\sigma(\delta))\H_1]$ for $\delta \in \Delta_G(G)$ to obtain a family of pairwise orthogonal projections summing to $1$. Observe that for $s \in G$, $\delta,\varepsilon \in \Delta_G(G)$, and $\xi \in \H_1$ we have
    \begin{align*}
        \pi(s)p_\delta \pi(\sigma(\varepsilon))\xi &= 1_{\{\delta\}}(\varepsilon)\pi(s \sigma(\varepsilon)) \xi \\
            &= 1_{\{\delta\}}(\varepsilon) \pi(\sigma(\Delta_G(s)\varepsilon)) \pi( \sigma(\Delta_G(s)\varepsilon)^{-1} s \sigma(\varepsilon)) \xi\\
            &= p_{\Delta_G(s)\delta} \pi(\sigma(\Delta_G(s)\varepsilon)) \pi( \sigma(\Delta_G(s)\varepsilon)^{-1} s \sigma(\varepsilon)) \xi = p_{\Delta_G(s)\delta}\pi(s) \pi(\sigma(\varepsilon))\xi.
    \end{align*}
This shows
    \begin{align}\label{eqn:comutation_relation_of_P_lambda_and_pi(G)}
        p_{\Delta_G(s)\delta} = \pi(s)p_{\delta} \pi(s)^* \qquad s\in G,\ \delta\in \Delta_G(G).
    \end{align}
Thus if we define a representation $U : K \to U(\H)$ by 
    \begin{align}\label{eqn:covariant_rep_of_point_modular_dual_group}
        U_\gamma := \sum_{\delta \in \Delta_G(G)} (\gamma \mid \delta) p_\delta\quad \gamma \in K,
    \end{align}
then $U_\gamma \pi(s) U_\gamma^* = (\gamma \mid \Delta_G(s)) \pi(s)$ holds for $\gamma \in K$ and $s \in G$.\\

\noindent\textbf{(ii)$\Rightarrow$(i):} Since $K$ is a compact group, $U$ is given by (\ref{eqn:covariant_rep_of_point_modular_dual_group}) for some family $\{p_\delta \in B(\H)\colon \delta\in \Delta_G(G)\}$ of pairwise orthogonal projections that sum to $1$ (see \cite[Theorem 4.44]{Fol16}). Since these projections form a partition of unity, there must exist at least one $\delta_0 \in \Delta_G(G)$ such that $p_{\delta_0}$ is non-zero. But then for $\xi \in p_{\delta_0}\H\setminus \{0\}$ and any $s\in G$, we have
    \[
        U_\gamma \pi(s) \xi = (\gamma \mid \Delta_G(s))\pi(s) U_\gamma \xi = (\gamma \mid \Delta_G(s))(\gamma \mid \delta_0)\pi(s) \xi =(\gamma \mid \Delta_G(s)\delta_0 )\pi(s) \xi.
    \]
Thus $\pi(s)\xi\in p_{\Delta_G(s)\delta_0} \H \setminus \{0\}$. By varying $s$ over $\{\sigma(\delta) \sigma(\delta_0)^{-1}\colon \delta\in \Delta_G(G)\}$, we see that $p_{\delta}$ is non-zero for all $\delta\in \Delta_G(G)$. Moreover, the above computation implies that (\ref{eqn:comutation_relation_of_P_lambda_and_pi(G)}) holds. In particular, $\pi(G_1)$ commutes with $p_\delta$ for all $\delta\in \Delta_G(G)$, and so each $(\pi|_{G_1}, p_{\delta}\H)$ gives a representation of $G_1$. Set $(\pi_1,\H_1) := (\pi|_{G_1}, p_1\H)$ and consider the unitary $W :\H \to \ell^2(\Delta_G(G)) \otimes \H_1$ defined by $[W\xi](\delta) = \pi(\sigma(\delta)^{-1}) p_{\delta}\xi$ for $\xi \in \H$ and $\delta\in \Delta_G(G)$. For $\xi \in \H$, $s\in G$, and $\delta \in \Delta_G(G)$, we compute using (\ref{eqn:comutation_relation_of_P_lambda_and_pi(G)}) that
    \begin{align*}
        [W\pi(s) \xi](\delta) = \pi(\sigma(\delta)^{-1} s) p_{\Delta_G(s)^{-1}\delta} \xi = \pi(\sigma(\delta)^{-1} s \sigma(\Delta_G(s)^{-1}\delta) ) [W\xi](\Delta_G(s)^{-1}\delta) =\left[(\Ind_{G_1}^G \pi_1)(s)W\xi\right](\delta).
    \end{align*}
Thus $(\pi,\H)\cong \Ind_{G_1}^G(\pi_1,\H_1)$.
\end{proof}

\subsection{Irreducible square integrable representations}

Let $(\pi,\H)$ be an irreducible representation of a locally compact group $G$. Recall that $(\pi, \H)$ is said to be \textit{square integrable} if there exists $\xi,\eta \in \H$ such that the function $c_{\xi,\eta}(s) := \langle \pi(s) \xi,\eta\rangle$, called a \textit{coefficient of} $\pi$, is a non-zero element in $L^2(G)$ (see the discussion following \cite[Theorem 2]{DM76}). By \cite[Theorem 2]{DM76}, $(\pi,\H)$ is subequivalent to the left regular representation of $G$. These conditions give rise to the \emph{formal degree operator} of $(\pi,\H)$, which is a unique non-zero, self-adjoint, positive operator $D$ on $\H$ that satisfies
    \[
        \pi(s)D \pi(s)^* = \Delta_G(s)^{-1} D \qquad s\in G.
    \]
The above property of $D$ is referred to as \emph{semi-invariance} with weight $\Delta_G^{-1}$ (see \cite[Section 1]{DM76}). One has $\xi \in \dom(D^{-1/2})$ if and only if $c_{\xi,\eta}\in L^2(G)$ for all $\eta\in \H$, and, moreover, the following orthogonality relation holds
    \begin{align}\label{eqn:formal_degree_operator_equation}
        \langle c_{\xi_1,\eta_1}, c_{\xi_2,\eta_2}\rangle_{L^2(\mu_G)} = \int_{G} \langle\pi(s) \xi_1,\eta_1\rangle \overline{\langle\pi(s) \xi_2,\eta_2\rangle} d\mu(s) = \langle D^{-\frac12}\xi_1,D^{-\frac12} \xi_2\rangle \langle \eta_2, \eta_1\rangle,
    \end{align}
for $\xi_1,\xi_2 \in \dom(D^{-1/2})$ and $ \eta_1,\eta_2 \in \H$ (see \cite[Theorem 3]{DM76}).

When $G$ is unimodular, the formal degree operator is a positive multiple of the identify operator $D=d_\pi 1$, and this constant $d_\pi$ is known as the \textit{formal degree} of $(\pi,\H)$. The following result shows that almost unimodularity tempers the formal degree operator and relates it directly to a formal degree of $\ker \Delta_G$.

\begin{thm}\label{thm:sq_int_irreps}
Let $(\pi,\H)$ be an irreducible square integrable representation of an almost unimodular group $G$. Then there exists an irreducible square integrable representation $(\pi_1,\H_1)$ of $\ker{\Delta_G}$ such that $(\pi,\H)\cong \Ind_{\ker{\Delta_G}}^G(\pi_1,\H_1)$, and the formal degree operator $D$ for $(\pi,\H)$ is diagonalizable with
    \[
        D= \sum_{\delta\in \Delta_G(G)} (d_{\pi_1}\delta) 1_{\{d_{\pi_1}\delta\}}(D),
    \]
where $d_{\pi_1}$ is the formal degree of $(\pi_1,\H_1)$.
\end{thm}
\begin{proof} Denote $G_1 : = \ker\Delta_G$ and let $\sigma\colon \Delta_G(G)\to G$ be a fixed normalized section for $\Delta_G$. The fact that $(\pi,\H)$ is induced by a representation $(\pi_1,\H_1)$ of $G_1$ follows from \cite[Proposition 7]{DM76}, and the irreducibility of $(\pi_1,\H_1)$ is standard (see, for example,  \cite[Corollary 2.43]{KT13}). Toward showing this representation is square integrable, let us identify $(\pi,\H)\cong \Ind_{G_1}^G(\pi_1,\H_1)$ so that one has
    \[
        \H = \bigoplus_{\delta \in \Delta_G(G)} \pi(\sigma(\delta))\H_1.
    \]
As in the proof of Theorem~\ref{thm:reps_induced_from_kernel}, set $p_\delta := [ \pi(\sigma(\delta)) \H_1]$. Recalling that (\ref{eqn:comutation_relation_of_P_lambda_and_pi(G)}) holds, it follows that
    \[
        \widetilde{D}:= \sum_{\delta\in \Delta_G(G)} \delta p_\delta
    \]
is semi-invariant with weight $\Delta_G^{-1}$, and therefore $D=c \widetilde{D}$ for some positive constant $c>0$ by \cite[Lemma 1]{DM76}. As this implies $\H_1 \subset \dom(D^{-1/2})$, we see that for any $\xi,\eta\in \H_1\setminus\{0\}$ one has
    \[
        0<\int_{G_1} |\<\pi(t)\xi,\eta\>|^2\ d\mu_{G}(t) \leq \int_G |\<\pi(s)\xi,\eta\>|^2\ d\mu_{G}(s) = \| c_{\xi,\eta}\|_{L^2(\mu_G)}^2<\infty.
    \]
(In fact, from (\ref{eqn:comutation_relation_of_P_lambda_and_pi(G)})
the two integrals are equal since $p_1 \pi(s) p_1 = 1_{G_1}(s) p_1 \pi(s) p_1$ for all $s\in G$.) Hence $(\pi_1,\H_1)$ is square integrable.

Finally, the claimed formula for $D$ will follow from showing the constant in $D= c \widetilde{D}$ is the formal degree $d_{\pi_1}$ of $(\pi_1,\H_1)$. Let $\xi,\eta\in \H_1$ be unit vectors. Then (\ref{eqn:formal_degree_operator_equation}) implies
    \[
        \frac{1}{c} = \| D^{-\frac12} \xi\| \|\eta\| = \int_G |\<\pi(s) \xi, \eta\>|^2\ d\mu_G(s) = \int_{G_1} |\<\pi(t) \xi, \eta\>|^2\ d\mu_G(t),
    \]
where the last equality follows from (\ref{eqn:comutation_relation_of_P_lambda_and_pi(G)}) as discussed above. This last expression equals $1/d_{\pi_1}$ by definition of the formal degree (see, for example, \cite[Section 3.3.a]{GdlHJ89}), and so $c=d_{\pi_1}$.
\end{proof}

\begin{rem}
Observe that the proof of the previous theorem implies that for fixed $\xi\in \H_1\setminus\{0\}$, the map
    \begin{align*}
        v\colon \H &\to L^2(G)\\
            \eta &\mapsto \frac{1}{d_{\pi_1}^{1/2} \|\xi\|} c_{\xi,\eta}
    \end{align*}
defines an isometry satisfying $v\pi(s) = \lambda_G(s) v$ for all $s\in G$. Moreover, for $\eta\in \pi(\sigma(\delta)) \H_1$, since $\pi(s)^*\eta \in \H_1$ if and only if $\Delta_G(s) = \delta$, it follows that $v p_\delta = 1_{\{\delta\}}(\Delta_G)v$ and therefore $vD = d_{\pi_1} \Delta_G$.$\hfill\blacksquare$
\end{rem}

\subsection{Factorial square integrable representations}
Recall that a representation $(\pi,\H)$ of a locally compact group $G$ is said to be \emph{factorial} (or a \emph{factor representation}) if $\pi(G)''$ is a factor. In \cite{Ros78}, Rosenberg defined the square integrability of such a representation as the existence of vectors $\xi,\eta \in \H$ whose coefficient $c_{\xi,\eta}(s)=\<\pi(s)\xi,\eta\>$ defines a non-zero element of $L^2(G)$. By \cite[Proposition 2.3.(a)]{Ros78}, this is equivalent to $(\pi,\H)$ being quasi-equivalent to a subrepresentation of the left regular representation of $G$; that is, $\pi(G)'' \cong L(G)p$ for some projection $p\in R(G)$. For $G$ second countable, Moore showed in \cite{Moo77} that this is further equivalent to the existence of a semi-invariant weight $\psi$ of degree $\Delta_G$ on the von Neumann algebra $\pi(G)''$ such that there exists some non-zero $x \in \dom(\psi)$ satisfying $s \mapsto \psi(\pi(s)x)$ is a square integrable function on $G$ (see \cite[Theorem 3]{Moo77}). Here, a non-zero normal semifinite weight $\psi$ on $\pi(G)''$ is said to be \emph{semi-invariant} of degree $\Delta_G$ if
    \[
        \psi(\pi(s) x\pi(s)^*) = \Delta_G(s) \psi(x) \qquad s \in G,\ x\in (\pi(G)'')_+.
    \]
By \cite[Proposition 2.1 and Theorem 1]{Moo77}, $\psi$ is faithful and unique up to scaling. By \cite[Theorem 4]{Moo77}, one can always normalize $\psi$ to satisfy
    \[
        \int_G \psi(\pi(s)x)\overline{\psi(\pi(s) y)}\ d\mu_{G}(s) = \psi(y^*x) \qquad x,y\in \dom(\psi),
    \]
and in this case $\psi$ is called the \emph{formal degree} of $(\pi,\H)$.

Note that for $G$ unimodular, the formal degree is a tracial weight. For almost unimodular groups, we have the following.

\begin{thm}\label{thm:formal_degree_is_almost_periodic}
Let $(\pi, \H)$ be a factorial square integrable representation of a second countable almost unimodular group $G$. Then the formal degree $\psi$ of $(\pi,\H)$ is almost periodic with $(\pi(G)'')^\psi=\pi(\ker{\Delta_G})''$.
\end{thm}
\begin{proof}
By \cite[Proposition 2.3.(a)]{Ros78}, $\pi(G)''\cong L(G)p$ for some projection $p\in R(G)$, where $\pi(s)\mapsto \lambda_G(s) p$ for all $s\in G$. Letting $z$ be the central support of $p$ in $R(G)$, there exists a further isomorphism $\theta\colon \pi(G)'' \to L(G)z$ satisfying $\theta(\pi(s)) = \lambda_G(s)z$ for all $s\in G$. Now $z\in R(G)\cap L(G)\subset L(G)^{\varphi_G}$ implies that the restriction of a Plancherel weight $\varphi_G$ to $L(G)z$ is a faithful normal almost periodic weight (see \cite[Remark 1.5]{GGLN25}). Thus $\psi:=\varphi_G\circ \theta$ is a faithful normal almost periodic weight on $\pi(G)''$ with
    \[
        \left(\pi(G)''\right)^\psi = \theta^{-1}\left( (L(G)z)^{\varphi_G} \right) = \theta^{-1}\left( L(G)^{\varphi_{G}} z\right)= \theta^{-1}(L(\ker{\Delta_G})z) = \pi(\ker{\Delta_G})''.
    \]
Additionally, by Remark~\ref{rem:Plancherel_weight_is_semi-invariant_weight_Delta_G} we have
    \[
        \psi(\pi(s) x \pi(s)^*) = \varphi_G( \lambda_G(s) \theta(x) \lambda_G(s)^*) = \Delta_G(s) \varphi_G( \theta(x)) = \Delta_G(s) \psi(x)
    \]
for all $s\in G$ and $x\in \pi(G)''_+$. Rescaling if necessary (which does not affect either the almost periodicity or the centralizer), it follows from \cite[Theorem 1]{Moo77} that $\psi$ is the formal degree of $(\pi,\H)$.
\end{proof}

We now prove our second main theorem.

\begin{thm}[{Theorem~\ref{introthm:B}}]\label{thm:sq_int_factor_reps}
Let $G$ be a second countable almost unimodular group and let $(\pi_1, \H_1)$ be a factorial square integrable representation of $\ker\Delta_G$. Then $\emph{Ind}_{\ker \Delta_G}^G(\pi_1,\H_1)$ is also factorial and square integrable.
\end{thm}
\begin{proof} Denote $G_1:= \ker\Delta_G$. By \cite[Proposition 2.3.(a)]{Ros78}, $\pi_1(G_1)''\cong L(G_1)q$ for some projection $q\in R(G_1)$, where $\pi(t)\mapsto \lambda_{G_1}(t) q$ for all $t\in G_1$. Although this isomorphism need not be spatial, by \cite[Theorem IV.5.5]{Tak02} there exists an ancillary Hilbert space $\K$ and an isometry $v\colon \H_1\to qL^2(G_1)\otimes \K$ satisfying $v\pi_1(t)=(\lambda_{G_1}(t)\otimes 1) v$ for all $t\in G_1$. Set $p:=vv^* \in qR(G_1)q \bar\otimes B(\mathcal{K})$. Thus, $(\pi_1, \H_1)$ and $ ((\lambda_{G_1}\otimes 1)p, p( L^2(G_1) \otimes \K))$ are unitarily equivalent. Consequently, $\text{Ind}_{G_1}^G(\pi_1,\H_1)$ and $\text{Ind}^G_{G_1}\left((\lambda_{G_1} \otimes 1)p, p(L^2(G_1) \otimes \K)\right)$ are unitarily equivalent.

Next we show that the above representations of $G$ are also unitarily equivalent to the representation $\left( (\lambda_G\otimes 1)p, p(L^2(G)\otimes \mathcal{K}) \right)$ of $G$. We identify $R(G_1)\leq R(G) = L(G)'$ and fix a normalized section $\sigma\colon \Delta_G(G)\to G$  for $\Delta_G$ and define a unitary $W \colon p(L^2(G) \otimes \K)\to \ell^2(\Delta_G(G))\otimes  p(L^2(G_1) \otimes \K)$ by
    \[
        [W pf](\delta) = p(1_{G_1} \otimes 1)(\lambda_G(\sigma(\delta))^* \otimes 1)f \qquad f\in L^2(G)\otimes \K\, , \delta\in \Delta_G(G).
    \]
Then $W$ intertwines $(\lambda_G \otimes 1)p$ and $\text{Ind}_{G_1}^G\left((\lambda_{G_1} \otimes 1)p\right)$. Indeed, for $s \in G$, $\delta \in \Delta_G(G)$, and $f \in L^2(G) \otimes \K$ we have
    \begin{align*}
        [\text{Ind}_{G_1}^G((\lambda_{G_1} \otimes 1)p)(s)Wpf](\delta) &= (\lambda_{G_1}(\sigma(\delta)^{-1}s\sigma(\Delta_G(s)^{-1}\delta)) \otimes 1)p [(Wpf)](\Delta_G(s)^{-1}\delta)\\
        &= (\lambda_{G_1}(\sigma(\delta)^{-1}s\sigma(\Delta_G(s)^{-1}\delta)) \otimes 1)p(1_{G_1} \otimes 1)(\lambda_G(\sigma(\Delta_G(s)^{-1}\delta))^* \otimes 1)f\\
        &=p(1_{G_1} \otimes 1)(\lambda_G(\sigma(\delta)^{-1}s) \otimes1 )f\\
        &=[Wp(\lambda_G(s) \otimes1 )f](\delta),
    \end{align*}
where the second-to-last equality follows by $p$ and $1_{G_1}\otimes 1$ commuting with $\lambda_{G_1}\otimes 1$. Thus, we set $(\pi,\H):=\Ind_{G_1}^G(\pi_1,\H_1)$ and we have a spatial isomorphism of $\pi(G)''\cong (L(G)\otimes \C)p$ that carries $\pi(s)$ to $(\lambda_G(s)\otimes 1)p$ for all $s\in G$. 

We shall show that $(L(G)\otimes \C)p$ is a factor by showing that the central support of $p$ in $qR(G)q \bar\otimes B(\mathcal{K})$ is minimal in the center. Towards that end, we need the following claim and provide its proof for completeness, though it may be well known to experts. 

Let $N \leq M$ be an inclusion of von Neumann algebras with $M\cap M' \subset N\cap N'$, and let $U$ a generating subgroup of the unitaries of $M$. Given a projection $r$ in $N$, we let $z_N$ and $z_M$ denote the central support of $r$ in $N$ and $M$, respectively. We claim that if $z_N$ is a minimal projection in $N \cap N'$, then $z_M$ is a minimal projection in $M \cap M'$. Indeed, for any central subprojection $z$ of $z_M$ in $M \cap M'$, we have that the central subprojection $zz_N \in N\cap N'$ of $z_N$ is zero or $z_N$ by the minimality of $z_N$. Thus either $z uz_Nu^*$ is equal to zero or $uz_Nu^*$ for all $u \in U$. Since 
    \[
        z_M = \bigvee_{u \in U} uz_Nu^*,
    \]
we have that $z$ is equal to zero or $z_M$. Thus $z_M$ is minimal in $M \cap M'$ as claimed. 

Since $\varphi_G$ is almost periodic and $L(G_1) \cong L(G)^{\varphi_G}$, we have $L(G) \cap R(G) \subset L(G_1) \cap R(G_1)$ (see \cite[Theorem 10]{Con72}). It follows that
    \[
        (R(G) \bar\otimes B(\K)) \cap (L(G) \otimes \C) = (R(G) \cap L(G)) \otimes \C \subset (R(G_1) \cap L(G_1)) \otimes \C = (R(G_1) \bar\otimes B(\K)) \cap (L(G_1) \otimes \C).
    \]
Additionally, this still holds if we compress by $q$: 
    \[
        (qR(G)q \bar\otimes B(\K)) \cap (L(G)q \otimes \C) \subset  (qR(G_1)q \bar\otimes B(\K)) \cap (L(G_1)q \otimes \C).
    \]    
We have that the central support of $p$ in $qR(G_1)q \bar\otimes B(\K)$ is minimal in the center since $\pi_1(G_1)$ is a factor. By the above claim, the central support $z$ of $p$ in $qR(G)q \bar\otimes B(\K)$ is minimal in the center and so $(L(G)\otimes \C)p \cong (L(G)\otimes \C)z $ is a factor. That is $(\pi, \H)$ is a factor representation. Furthermore, following the proof of Theorem \ref{thm:sq_int_factor_reps}, we can define a formal degree $\psi$ on $\pi(G)''$. Hence $(\pi, \H)$ is a factorial square integrable representation.
\end{proof}

\begin{rem}
For a factorial square integrable representation $(\pi_1, \H_1)$ of $\ker\Delta_G$, let $(\pi,\H) = \text{Ind}_{\ker \Delta_G}^G(\pi_1,\H_1)$. By Theorem~\ref{thm:formal_degree_is_almost_periodic}, $\pi(G)''$ admits an almost periodic formal degree $\psi$ with centralizer given by $\pi(\ker\Delta_G)''$. One might hope that $\pi(\ker\Delta_G)''\cong \pi_1(\ker\Delta_G)''$ (i.e. that $\pi|_{\ker\Delta_G}$ and $\pi_1$ are quasi-equivalent), so that $\psi$ is an extremal weight. However, this need not be the case. For example, if $(\pi,\H)$ is irreducible, then it is known that $\psi = \Tr(D^{-1/2} \cdot D^{-1/2})$ where $D$ is the formal degree operator of $(\pi,\H)$ (see \cite[Theorem 2]{Moo77}). In this case, the centralizer is isomorphic to $\pi_1(\ker\Delta_G)''\otimes \ell^\infty(\Delta_G(G))$. More generally, for $p$ as in the proof of Theorem~\ref{thm:sq_int_factor_reps}, one will have $\pi(\ker\Delta_G)''$ is factor if and only if the central support of $p$ in $R(\ker\Delta_G)\bar\otimes B(\K)$ is central in $R(G)\bar\otimes B(\K)$.$\hfill\blacksquare$
\end{rem}

\section{Group von Neumann Algebra Properties}\label{sec:group_vNas}

In this section, we study the group von Neumann algebra $L(G)$ associated to an almost unimodular group $G$. In the first subsection, we establish some alternate presentations of the basic construction associated to the inclusion $L(\ker{\Delta_G})\leq L(G)$ (see Theorem~\ref{thm:basic_construction_crossed_product}), which we use to relate the factoriality of $L(\ker{\Delta_G})$ and $L(G)$. The results of this subsection will also be used in the final subsection, where we prove Theorem~\ref{introthm:D} (see Theorem~\ref{thm:formula}). In the intermediate subsection, we show all intermediate von Neumann algebras $L(\ker{\Delta_G})\leq P\leq L(G)$ are of the form $P=L(H)$ for some closed intermediate group $\ker{\Delta_G}\leq H\leq G$ if and only if $L(\ker{\Delta_G})$ is a factor (see Theorem~\ref{thm:intermediate_algebras})

Throughout this section, $(\cdot\mid\cdot)$ will always denote dual pairings between locally compact abelian groups, we will make use of the following terminology. For a locally compact group $G$, let $\hat{\iota}\colon \R\to \Delta_G(G)\hat{\ }$ be the map dual to the inclusion map $\iota\colon \Delta_G(G) \hookrightarrow \R_+$; that is,
    \begin{equation}\label{eqn:transpose_formula}
        (\hat{\iota}(t)\mid \delta) = ( t \mid  \iota(\delta)) = \delta^{it} \qquad t\in \R,\ \delta\in \Delta_G(G).
    \end{equation}
If $G$ is almost unimodular so that $\varphi_G$ is almost periodic, then the modular automorphism group $\sigma^{\varphi_G}\colon \R\curvearrowright L(G)$ admits an extension $\alpha\colon \Delta_G(G)\hat{\ } \curvearrowright L(G)$ satisfying $\alpha_{\hat{\iota}(t)} = \sigma_t^{\varphi_G}$ for all $t\in \R$ (see \cite[Section 1.4]{GGLN25} for more details). We will refer to this action of $\Delta_G(G)\hat{\ }$ as the \emph{point modular extension} of the modular automorphism group.

\subsection{The basic construction and factoriality}

For an almost unimodular group $G$, the strict semifiniteness of a Plancherel weight $\varphi_G$ (guaranteed by Theorem~\ref{thm:when_is_plancherel_weight_almost_periodic}) implies that $L^2(L(\ker{\Delta_G}),\varphi_G)$ forms a non-trivial closed subspace of $L^2(L(G),\varphi_G)$. Letting $e_{\varphi_G}$ denote the projection onto this subspace, the von Neumann algebra $\<L(G),e_{\varphi_G}\>$ generated by $L(G)$ and $e_{\varphi_G}$ is the basic construction for the inclusion $L(\ker{\Delta_G})\leq L(G)$. We begin by showing this von Neumann algebra has several natural presentations, and in the proof it will be helpful to recall that $e_{\varphi_G}$ simply corresponds to $1_{\ker{\Delta_G}}$ under the usual identification $L^2(G,\varphi_G)\cong L^2(G)$. Note also that the following theorem partially resolves a question of Sutherland from \cite{Sut78} (see Remark (ii) following the proof of Theorem 3.1).

\begin{thm}\label{thm:basic_construction_crossed_product}
Let $G$ be an almost unimodular group, let $\varphi_G$ be a Plancherel weight on $L(G)$, and let $\alpha\colon \Delta_G(G)\hat{\ }\curvearrowright L(G)$ be the point modular extension of $\sigma^{\varphi_G}\colon \R\curvearrowright L(G)$. Then one has
    \[
        \<L(G), e_{\varphi_G}\> \cong L(G)\rtimes_\alpha  \Delta_G(G)\hat{\ } \cong L(\ker{\Delta_G})\bar\otimes B(\ell^2\Delta_G(G)).
    \]
\end{thm}
\begin{proof}
Denote $K:=\Delta_G(G) \hat{\ }$ for convenience. For each $\gamma\in K$, define
    \[
        U_\gamma:= \sum_{\delta\in \Delta_G(G)} (\gamma\mid \delta) 1_{\Delta_G^{-1}(\{\delta\})}.
    \]
Then, one has $\alpha_\gamma(x) = U_\gamma x U_\gamma^*$ for all $\gamma\in K$ and $x\in L(G)$, and
    \[
        \< L(G), e_{\varphi_G}\> = L(G)\vee \{U_\gamma\colon \gamma\in K\}''.
    \]
Additionally, $\beta_\gamma(y):= U_\gamma y U_\gamma^*$ defines an action $\beta\colon K \curvearrowright R(G)$ satisfying $\beta_\gamma(\rho_G(s)) = \overline{(\gamma\mid \Delta_G(s))} \rho_G(s)$ for all $\gamma\in K$ and $s\in G$. Consider the crossed product $R(G)\rtimes_\beta K$ and the family of projections
    \[
        p_\delta:= \int_{K} (\gamma\mid \delta) \lambda_K(\gamma)\ d\mu_K(\gamma) \qquad \delta\in \Delta_G(G),
    \]
where $\mu_K$ is the unique Haar measure on $K$ satisfying $\mu_K(K)=1$. By \cite[Theorem 2.2]{Hag76}, to establish the first claimed isomorphism it suffices to show that $p_1$ has full central support in $R(G)\rtimes_\beta K$. Observe that $\lambda_K(\gamma) \rho_G(s) \lambda_K(\gamma)^* = \overline{(\gamma\mid\Delta_G(s))} \rho_G(s)$ implies $\rho_G(s)\lambda_K(\gamma) \rho_G(s)^* = (\gamma\mid \Delta_G(s)) \lambda_K(\gamma)$. Consequently,
    \[
        \rho_G(s) p_1 \rho_G(s)^* = \int_K (\gamma\mid \Delta_G(s)) \lambda_K(\gamma)\ d\mu_K(\gamma) = p_{\Delta_G(s)}.
    \]
Hence the central support of $p_1$ is at least $\sum_\delta p_\delta = 1$.

For the second isomorphism, we recall that we can identify
    \[
        G \cong \Delta_G(G) \sltimes{(\beta,c)}\ker{\Delta_G}
    \]
for some cocycle action $(\beta,c)\colon \Delta_G(G)\curvearrowright \ker{\Delta_G}$ (see Example~\ref{ex:cocycle_realization_of_almost_unimodular_groups}). Defining $(\check{\alpha},\check{c})\colon \Delta_G(G) \curvearrowright L(\ker{\Delta_G})$ by 
    \[
       \check{\alpha}_\delta\left(\lambda_{G}(s)\right) := \lambda_{G}(\beta_\delta(s)) \qquad \text{ and } \qquad \check{c}(\delta_1, \delta_2) : = \lambda_{G}(c(\delta_1,\delta_2)),
    \]
it follows that
    \[
        L(G)\cong L(\ker{\Delta_G}) \rtimes_{(\check{\alpha},\check{c})} \Delta_G(G)
    \]
and $\alpha$ is dual to $(\check{\alpha},c)$ (see \cite[Proposition 3.1.7]{Sut80}). Thus 
    \[
        L(G)\rtimes_\alpha K \cong \left[ L(\ker{\Delta_G}) \rtimes_{(\check{\alpha},\check{c})} \Delta_G(G) \right] \rtimes_\alpha K \cong L(\ker{\Delta_G}) \bar\otimes B(\ell^2 \Delta_G(G))
    \]
by \cite[Theorem 2]{NS79}.
\end{proof}

% \begin{rem}
% Furthermore, $L(\ker{\Delta_G})$ is injective if and only if $L(G)$ is injective (see \cite[Theorem 6.2]{Sut80}).
% \end{rem}

\begin{cor}\label{cor:factoriality}
For an almost unimodular group $G$, the following are equivalent:
    \begin{enumerate}[label=(\roman*)]
        \item $L(\ker{\Delta_G})$ is a factor;
        \item $L(G)$ is a factor and $z\lambda_G(s)z\neq 0$ for all $s\in G$ and non-zero projections $z\in L(\ker{\Delta_G})'\cap L(\ker{\Delta_G})$.
    \end{enumerate}
In this case, $L(G)$ is: semifinite if $\Delta_G(G)=\{1\}$; type $\mathrm{III}_\lambda$ if $\Delta_G(G)=\lambda^\Z$ for some $0< \lambda <1$; and type $\mathrm{III}_1$ if $\Delta_G(G)$ is dense in $\R_+$.
\end{cor}
\begin{proof}
Throughout we fix a Plancherel weight $\varphi_G$ on $L(G)$.\\

\noindent\textbf{($\Rightarrow$):} Identifying $L(\ker{\Delta_G})\cong L(G)^{\varphi_G}$ by Theorem~\ref{thm:when_is_plancherel_weight_almost_periodic}, we see that the factoriality of $L(G)$ follows from $L(G)'\cap L(G)\subset L(\ker{\Delta_G})'\cap L(G)=\C$ (see \cite[Theorem 10]{Con72}). Also, the condition $z\lambda_G(s)z\neq 0$ holds all $s\in G$ and the only non-trivial central projection $z=1$ in $L(\ker{\Delta_G})$.\\

\noindent\textbf{($\Leftarrow$):} Let $\alpha\colon \Delta_G(G)\hat{\ }\curvearrowright L(G)$ be the point modular extension of $\sigma^{\varphi_G}$. Then our assumption implies that for all $\delta\in \Delta_G(G)$ there exists a non-zero operator $x\in zL(G)z$ satisfying $\alpha_\gamma(x) = (\gamma\mid \delta) x$ for all $\gamma\in \Delta_G(G)\hat{\ }$. Consequently, the Arveson spectrum of the restricted action $\alpha^z\colon \Delta_G(G)\hat{\ }\curvearrowright zL(G)z$ is $\Delta_G(G)$ (see \cite[Lemma XI.1.12]{Tak03}). It follows that the Connes spectrum is $\Gamma(\alpha)= \Delta_G(G)$ and hence $L(G)\rtimes_{\alpha} \Delta_G(G)\hat{\ }$ is a factor (see \cite[Lemma XI.2.2 and Corollary XI.2.8]{Tak03}). Using Theorem~\ref{thm:basic_construction_crossed_product}, we can identify this crossed product with $L(\ker{\Delta_G}) \bar\otimes B(\ell^2\Delta_G(G))$ and see that $L(\ker{\Delta_G})$ must be a factor.\\

\noindent For the final observation, it suffices to show that the modular spectrum $\S(L(G))$ is given by the closure of $\Delta_G(G)$ in $[0,+\infty)$ (see \cite[Theorem XII.1.6]{Tak03}). First note that $\S(L(G)) \subset \overline{\Delta_G(G)}$ since the latter set gives the spectrum of $\Delta_{\varphi_G}$ by Theorem~\ref{thm:when_is_plancherel_weight_almost_periodic}. For the reverse inclusion, note that each $\lambda_G(s)$ normalizes $L(\ker{\Delta_G})$ and satisfies $\varphi_G(\lambda_G(s)\,\cdot\, \lambda_G(s)^*) = \Delta_G(s) \varphi_G$ for $s \in G$. Consequently $\Delta_G(G) \in \S(L(G))$ by \cite[Theorem 3.3.1]{Con73}.
\end{proof}

\begin{exs}\label{ex:interesting_group_vNas}
    \begin{enumerate}
    \item[]

    \item For a prime number $p\in \N$, let $G_p:=\Q_p^\times \ltimes \Q_p$ be as in Example~\ref{exs:first_examples}.(\ref{ex:p-adic_ax+b_group}) so that $\ker{\Delta_{G_p}}=\Z_p^\times \ltimes \Q_p$ and $\Delta_{G_p}(G_p)=p^\Z$. It is known that $L(G_p)$ is a type $\mathrm{I}_\infty$ factor (see the discussion at the beginning of \cite[Section 3]{Bla77}), whereas $L(\ker{\Delta_{G_p}})\cong L(\Q_p)\rtimes \Z_p^\times$ is not a factor since, for example, $\lambda_{\Q_p}(1_{\Z_p})$ is non-trivial central element. (Alternatively, the factoriality of $L(\ker{\Delta_{G_p}})=L(G_p)^{\varphi_{G_p}}$ would imply $\varphi_{G_p}$ is tracial by \cite[Remark 1.1]{GGLN25} contradicting $G_p$ being non-unimodular.) Thus the final statement in Corollary~\ref{cor:factoriality} can fail without factoriality of $L(\ker{\Delta_G})$.

    Now, let $(p_n)_{n \in \N}$ be a sequence of (not necessarily distinct) primes. For each $n \in \N$,  consider the compact open subgroup $K_{p_n} := \Z_{p_n}^\times \ltimes \Z_{p_n} \leq G_{p_n}$. Recall from Example~\ref{ex:interesting_group_vNas}.(\ref{ex:restricted_direct_products}) that the restricted direct product
        \[
            G := \prod_{n \in \N}(G_{p_n}, K_{p_n})
        \]
    is an almost unimodular group. By \cite[Theorem 4.1]{Bla77}, $L(G)$ is a type $I$ factor if $\sum_{n \in \N} p_n^{-1} < \infty$ or a type $\mathrm{III}$ ITPFI factor if $\sum_{n \in \N} p_n^{-1}=\infty$. In particular, if $p_n = p$ for all $n \in \N$ then $L(G)$ is type $\mathrm{III}_{\frac1p}$ (see \cite[Theorem 4.2]{Bla77}), and if $p_n$ denotes the $n$th prime then $L(G)$ is type $\mathrm{III}_1$ (see \cite[Theorem 4.4]{Bla77} and \cite[Theorem 2.10]{BZ00}). In the former case, the modular automorphism group $\sigma^{\varphi_G}$ is $\frac{2\pi}{\log(p)}$-periodic and consequently $L(\ker{\Delta_G})\cong L(G)^{\varphi_G}$ is a factor by \cite[Theorem 4.2.6]{Con73}. In the latter case, and more generally when the sequence $(p_n)_{n\in \N}$ consists of distinct primes, one has
        \[
            \ker{\Delta_G} = \prod_{n\in \N} ( \ker{\Delta_{G_{p_n}}}, K_{p_n}),
        \]
    and consequently $L(\ker{\Delta_G})$ is not factor.

    \item\label{ex:matrix_action_non-injective} Let $\alpha \colon \text{GL}_2(\R) \curvearrowright \R^2$ be the action by matrix multiplication. Consider a countable intermediate subgroup $\text{SL}_2(\Z)\leq H\leq \text{GL}_2(\R)$. Restricting $\alpha$ to $H$, for $G:= H \sltimes{\alpha} \R^2$ we have
        \[
            \Delta_{G}(G) = |\det(H)| \qquad \text{ and }\qquad \ker{\Delta_G} = (H \cap \text{det}^{-1}(\{ \pm1\})) \sltimes{\alpha} \R^2.
        \]
    Thus $G$ is almost unimodular by either of Propositions~\ref{prop:short_exact_sequences} or \ref{prop:cocycle_semidirect_products}. Since the action of $H\cap \text{det}^{-1}(\{ \pm1\})$ on $\R^2$ is essentially free and ergodic---inheriting the former property from the action of $\text{SL}_2(\R)$ and the latter from $\text{SL}_2(\Z)$ (see \cite[Corollary of Theorem 6]{Moo66})---it follows from \cite[Theorem XIII.1.7]{TakIII} that $L(\ker{\Delta_{G}})$ is a separable type $\mathrm{II}_\infty$ factor. Moreover, $L(\ker{\Delta_G})$ is non-injective since it contains a copy of $L(\text{SL}_2(\Z))$. By Corollary~\ref{cor:factoriality} we have that $L(G)$ is a separable non-injective factor of type: $\mathrm{II}_\infty$ if $\det(H)=\{1\}$;  $\mathrm{III}_\lambda$ if $\det(H)=\lambda^\Z$ for some $0< \lambda <1$; and $\mathrm{III}_1$ if $\det(H)$ is dense in $\R_+$. The special cases $H=\text{GL}_2(\Q)$ and $H=\text{SL}_2(\Q)\vee \lambda^\Z$ were previously considered by Godement in \cite{God51} and Sutherland in \cite[Section 5]{Sut78}, respectively.
    
    \item\label{ex:matrix_action_injectivity} Let $\text{UT}_2(\R)$ denote the upper triangular matrices with real entries and let $\alpha \colon \text{UT}_2(\R) \curvearrowright \R^2$ be the action by matrix multiplication. Consider a countable intermediate subgroup $\text{N}_2(\Q)\leq H\leq \text{UT}_2(\R)$, where $\text{N}_2(\Q)=\text{UT}_2(\R) \cap \text{SL}_2(\Q)$. Restricting $\alpha$ to $H$, for $G:= H \sltimes{\alpha} \R^2$ we have
        \[
            \Delta_{G}(G) = |\det(H)| \qquad \text{ and }\qquad \ker{\Delta_{G}} = (H \cap \text{det}^{-1}(\{ \pm1\})) \sltimes{\alpha} \R^2.
        \]
    Thus $G$ is an almost unimodular group by either of Propositions~\ref{prop:short_exact_sequences} or \ref{prop:cocycle_semidirect_products}. Using that $\text{N}_2(\Q)$ acts ergodically on $\R^2$ (see the proof of \cite[Lemma 5.2]{Sut78}), it follows by the same argument as in the previous example that $L(\ker{\Delta_{G}})$ is a separable type $\mathrm{II}_\infty$ factor. Moreover, $H$ is solvable as a subgroup of $\text{UT}_2(\R)$, and so arguing as in the proof of \cite[Lemma 5.2]{Sut78} we see that $L(G)$, and hence $L(\ker{\Delta_G})$, is injective. Then by Corollary~\ref{cor:factoriality} we have that $L(G)$ is the unique separable injective factor of type: $\mathrm{II}_\infty$ if $\det(H)=\{1\}$; $\mathrm{III}_\lambda$ if $\det(H)=\lambda^\Z$ for some $0< \lambda <1$; and $\mathrm{III}_1$ if $\det(H)$ is dense in $\R_+$. The special cases $H=\text{UT}_2(\Q)$ and $H=\text{N}_2(\Q)\vee \lambda^\Z$ were previously considered by Sutherland in \cite[Section 5]{Sut78}. 

    \item For non-zero integers $m,n\in \Z\setminus \{0\}$, consider the Baumslag--Solitar group
        \[
            \text{BS}(m,n):=\<a,t\mid ta^m t^{-1} = a^n\>,
        \]
    and the commensurated subgroup $\<a\>$. Let $\text{G}(m,n)$ denote the \emph{relative profinite completion} (or \emph{Schlichting completion}) of $\text{BS}(m,n)$ with respect to $\<a\>$, which is defined as the closure of the representation of $\text{BS}(m,n)$ in $\text{Sym}(\text{BS}(m,n)/\<a\>)$ acting by left multiplication. These groups are totally disconnected by \cite{EW18} and hence almost unimodular by Example~\ref{exs:first_examples}.(\ref{ex:totally_disconnected}). One has $\Delta_{G(m,n)}(G(m,n)) = |\frac{m}{n}|^\Z$ (see \cite[Lemma 9.1]{Rau19}), and, in particular, $\Delta_{G(m,n)}(a)=1$ and $\Delta_{G(m,n)}(t)=\left|\frac{m}{n}\right|$. Consequently $\ker{\Delta_{G(m,n)}}$ is either all of $G(m,n)$ if $|n|=|m|$ and otherwise is the closure of the subgroup of $\text{BS}(m,n)$ consisting of words with the same number of $t$'s as $t^{-1}$'s. For $2\leq |m| < n$, we have that $L(G(m,n))$ is a non-injective type $\mathrm{III}_{|\frac{m}{n}|}$ factor by \cite[Theorem 9.2]{Rau19} (see also \cite{RauErr21} and \cite{Suz17}). In this case, one also has that the modular automorphism group $\sigma^{\varphi_{G(m,n)}}$ is $\frac{2\pi}{\log(|n|/|m|)}$-periodic and consequently $L(\ker{\Delta_{G(m,n)}})\cong L(G(m,n))^{\varphi_{G(m,n)}}$ is a factor by \cite[Theorem 4.2.6]{Con73}.$\hfill\blacksquare$
    \end{enumerate}
\end{exs}

Our last result in this subsection will be needed for our generalization of the Atiyah--Schmid formula in Theorem~\ref{thm:formula}.

\begin{thm}\label{thm:extending_reps_to_basic_construction}
Let $G$ be an almost unimodular group $G$ and fix a Plancherel weight $\varphi_G$ on $L(G)$. For a representation $(\pi, \H)$ of $L(G)$, the following are equivalent:
    \begin{enumerate}[label=(\roman*)]
        \item There is a representation $\widetilde{\pi}\colon \<L(G), e_{\varphi_G}\>\to B(\H)$ extending $\pi$.
        \item There exists a representation $U\colon \Delta_G(G)\hat{\ }\to U(\H)$ satisfying $U_\gamma \pi(\lambda_G(s)) U_\gamma^* = (\gamma \mid \Delta_G(s)) \pi(\lambda_G(s))$.
    \end{enumerate}
In particular, given either of the above representations, there is a unique other representation satisfying
    \begin{align}\label{eqn:Jones_projection_image}
        \widetilde{\pi}(e_{\varphi_G}) = \int_{\Delta_G(G) \hat{\ }} U_\gamma d\mu_{\Delta_G(G) \hat{\ }}(\gamma),
    \end{align}
where $\mu_{\Delta_G(G) \hat{\ }}$ is the unique Haar probability measure on $\Delta_G(G) \hat{\ }$.
\end{thm}
\begin{proof}
Once again we denote $K:= \Delta_G(G)\hat{\ }$ for convenience.\\

\noindent \textbf{(i)$\Rightarrow$(ii):} For each $\gamma\in K$ define a unitary on $\H$ by
    \[
        U_\gamma:= \sum_{\delta \in\Delta_G(G)} (\gamma\mid\delta) \widetilde{\pi}\left(1_{\Delta_G^{-1}(\{\delta\})}\right).
    \]
Then $\gamma\mapsto U_\gamma$ has the desired properties. In particular, if $\mu_K$ is the unique Haar probability measure on $K$ then for each $s\in G$ one has
    \begin{align*}
        \int_{K} U_\gamma\ d\mu_K(\gamma) &= \sum_{\delta\in \Delta_G(G)} \int_K (\gamma\mid\delta)\ d\mu_K(\gamma) \widetilde{\pi}\left(1_{\Delta_G^{-1}(\{\delta\})}\right)\\
            &= \widetilde{\pi}\left(1_{\Delta_G^{-1}(\{1\})}\right) = \widetilde{\pi}(e_{\varphi_G}).
    \end{align*}
If $V\colon K\to U(\H)$ is another representation satisfying the covariance condition and (\ref{eqn:Jones_projection_image}), then the latter implies its $1$-eigenspace is $\widetilde{\pi}(e_{\varphi_G})$. The other $\delta$-eigenspaces, for $\delta\in \Delta_G(G)$, are then determined by the covariance condition with $\pi$, which forces $V=U$.\\

\noindent \textbf{(ii)$\Rightarrow$(i):} We essentially follow the proof of \cite[Proposition 2.25]{GGLN25}, with unitarity of our eigenoperators replacing the need for factoriality. Let $\hat{\iota}\colon \R\to K$ be the transpose of the inclusion map $\iota\colon \Delta_G(G) \hookrightarrow \R_+$ defined by (\ref{eqn:transpose_formula}). Then $\R\ni t\mapsto W_t:= U_{\hat{\iota}(t)}$ defines a representation of $\R$ on $\H$ satisfying
    \[
        W_t \pi(\lambda_G(s)) W_t^* = \Delta_G(s)^{it} \pi(\lambda_G(s)) = \pi( \sigma_t^{\varphi_G}(\lambda_G(s))).
    \]
Using normality of conjugation by $W_t$ and the modular automorphism group of $\sigma_t^{\varphi_G}$, for fixed $t\in \R$, it follows that $W_t \pi(x) W_t^* = \pi(\sigma_t^{\varphi_G}(x))$ for all $x\in L(G)$. Thus by \cite[Proposition 2.4]{GGLN25}, it suffices to show that $\pi(L(G)) \H^W \leq \H$ is dense. Noting that $\H^W = \H^U$ by density of $\hat{\iota}(\R)\leq K$, it suffices to show $\pi(L(G)) \H^U$ is dense. Since $K$ is compact, for each $\gamma\in K$ we have
    \[
        U_\gamma = \sum_{\delta\in \Delta_G(G)} (\gamma\mid\delta) p_\delta
    \]
for a family of pairwise orthogonal projections $\{p_\delta \in B(\H)\colon \delta\in \Delta_G(G)\}$ that sum to one and are given explicitly by
    \[
        p_\delta = \int_K \overline{(\gamma \mid \delta)} U_\gamma\ d\mu_K(\gamma)
    \]
(see \cite[Theorem 4.44]{Fol16}). Using the assumed conjugation relation between $U_\gamma$ and $\pi(\lambda_G(s))$ one obtains
    \[
        \pi(\lambda_G(s)) p_1 \pi(\lambda_G(s))^* = p_{\Delta_G(s)}
    \]
for all $s\in G$ by arguing as in the proof of Theorem~\ref{thm:basic_construction_crossed_product}. Consequently,
    \[
        p_{\Delta_G(s)} \H =  \pi(\lambda_G(s)) p_1 \H = \pi(\lambda_G(s)) \H^U \subset \pi(L(G)) \H^U.
    \]
Since the $p_\delta$ sum to one, this establishes the needed density. Note that the representation $\widetilde{\pi}\colon \<L(G),e_{\varphi_G}\>\to B(\H)$ obtained from \cite[Proposition 2.4]{GGLN25} satisfies $\widetilde{\pi}(\Delta_{\varphi_G}^{it}) = W_t$, and consequently $\widetilde{\pi}(e_{\varphi_G})$ is given by the projection onto $\H^W=\H^U$, which is $p_1$. This also uniquely determines $\widetilde{\pi}$ since it extends $\pi$.
\end{proof}

\subsection{Intermediate von Neumann algebras}

Each closed intermediate group $\ker{\Delta_G} \leq H \leq G$ yields an intermediate von Neumann algebra $L(\ker{\Delta_G})\leq L(H)\leq L(G)$, and so it is natural to ask if this accounts for \emph{all} intermediate algebras. For second countable almost unimodular groups this turns out to be the case if (and only if) $L(\ker{\Delta_G})$ is a factor (see Theorem~\ref{thm:intermediate_algebras} below). Ultimately, this will follow from \cite{ILP98}, and so we must first witness $L(H)$ as the fixed point subalgebra of some minimal action. Toward this end, recall from Remark~\ref{rem:Galois_correspondence} that each intermediate group $\ker{\Delta_G}\leq H\leq G$ can be identified with
    \[
        H\cong \Delta_G(H) \sltimes{(\beta,c)} \ker{\Delta_G}
    \]
for a continuous cocycle action $(\beta,c)\colon \Delta_G(G)\curvearrowright \ker{\Delta_G}$. Letting $(\check{\alpha}, \check{c})\colon \Delta_G(G) \curvearrowright L(\ker{\Delta_G})$ be as in the proof of Theorem~\ref{thm:basic_construction_crossed_product}, one then has
    \[
        L(H)\cong L(\ker{\Delta_G}) \rtimes_{(\check{\alpha},\check{c})} \Delta_G(H).
    \]
Let $\alpha\colon \Delta_G(G)\hat{\ }\curvearrowright L(G)$ be the point modular extension of the modular automorphism group of $\varphi_G$. Then $\alpha$ is also the dual action to $(\check{\alpha}, \check{c})$, and if $\Delta_G(H)^\perp := \{\gamma\in \Delta_G(G)\hat{\ }\colon (\gamma\mid \delta)=1 \ \forall \delta\in \Delta_G(H)\}$ then it follows that
    \[
        L(H)\cong \{x\in L(G) \colon \alpha_\gamma(x) = x \ \forall \gamma\in \Delta_G(H)^\perp \} =: L(G)^{\Delta_G(H)^\perp}.
    \]
Conversely, for a closed subgroup $K\leq \Delta_G(G)\hat{\ }$ we have
    \begin{align}\label{eqn:fixed_point_crossed_product_duality}
        L(G)^K \cong L( K_\perp \sltimes{(\beta,c)} \ker{\Delta_G})
    \end{align}
where $K_\perp := \{\delta\in \Delta_G(G)\colon (\gamma\mid \delta)=1 \ \forall \gamma\in K\}$. Thus the problem of determining which intermediate algebras are of the form $L(H)$ is equivalent to determining which are of the form $L(G)^K$, and we first give a characterization of the latter in full generality.

\begin{prop}\label{prop:characterizing_fixed_point_intermediate_algebras}
Let $G$ be an almost unimodular group, let $\varphi_G$ be a Plancherel weight on $L(G)$, and let $\alpha\colon \Delta_G(G)\hat{\ }\curvearrowright L(G)$ be the point modular extension of the modular automorphism group of $\varphi_G$. For an $\alpha$-invariant intermediate von Neumann algebra $L(\ker{\Delta_G})\leq P\leq L(G)$, denote for each $\delta\in \Delta_G(G)$
    \[
        z_\delta:=\bigvee \{vv^*\colon v\in P \text{ partial isometry with }\alpha_\gamma(v) = (\gamma \mid \delta) v \ \forall \gamma\in \Delta_G(G)\hat{\ } \}.
    \]
Then one has $P=L(G)^K$ for some closed subgroup $K\leq \Delta_G(G)\hat{\ }$ if and only if $z_\delta = 1_{\Sd(\varphi_G|_P)}(\delta)$ for all $\delta\in \Delta_G(G)$. In this case, $\Sd(\varphi_G|_P)$ is a group and one has $K=\Sd(\varphi_G|_P)^\perp$.
\end{prop}
\begin{proof}
First note that $\varphi_G$ is semifinite on $P$ since it is semifinite on $L(\ker{\Delta_G})$. Additionally, $P$ is $\sigma^{\varphi_G}$-invariant since $\alpha$ extends $\sigma^{\varphi_G}$, and consequently $L^2(P,\varphi_G) \leq L^2(L(G),\varphi_G)$ is an invariant subspace for $\Delta_{\varphi_G}$. It follows that $\varphi_G|_{P}$ is almost periodic, and thus $z_\delta=0$ for $\delta\not\in \Sd(\varphi_G|_P)$ and otherwise is a projection in the center of $L(\ker{\Delta_G})$ by \cite[Lemma 2.1]{GGLN25}.

Now, first suppose that $P=L(G)^K$ for some closed subgroup $K\leq \Delta_G(G)\hat{\ }$. Then $\Sd(\varphi_G|_P)=K_\perp$ by (\ref{eqn:fixed_point_crossed_product_duality}) and so $z_\delta=0$ for all $\delta\not\in K_\perp$ by the first part of the proof. For $\delta\in K_\perp$, let $s\in \Delta_G^{-1}(\{\delta\})$ so that $\lambda_G(s)\in L(G)^K=P$. Consequently, $z_{\delta} \geq \lambda_G(s)\lambda_G(s)^*=1$, and therefore $z_\delta =1$. Also note that in this case one has $K=(K_\perp)^\perp =\Sd(\varphi_G|_P)^\perp$ as claimed.

Conversely, suppose $z_\delta=1_{\Sd(\varphi_G|_P)}(\delta)$. For each $\delta\in \Sd(\varphi_G|_P)$, \cite[Lemma 2.1]{GGLN25} allows us to write
    \[
        \sum_{v\in \mathcal{V}_\delta} vv^* = z_\delta =1
    \]
for a family of partial isometries $\mathcal{V}_\delta\subset P$ satisfying $\sigma_t^{\varphi_G}(v) = \delta^{it} v$ for all $t\in \R$ and $v\in \mathcal{V}_\delta$. Then for any $s\in \Delta_G^{-1}(\{\delta\})$ one has $\lambda_G(s)^* v \in L(G)^{\varphi_G} = L(\ker{\Delta_G})$ for all $v\in \mathcal{V}_\delta$, and thus
    \[
        \lambda_G(s)^* = \sum_{v\in \mathcal{V}_\delta} (\lambda_G(s)^* v) v^* \subset L(G)^{\varphi_G} P = P.
    \]
Thus $\lambda_G(\Delta_G^{-1}(\Sd(\varphi_G|_P))) \subset P$. Observe that this implies $\Sd(\varphi_G|_P)$ is a group since for $\delta_1,\delta_2\in \Sd(\varphi_G|_P)$, if $s_i\in \Delta_G^{-1}(\{\delta_i\})$ for $i=1,2$ then $\lambda_G(s_1s_2)=\lambda_G(s_1)\lambda_G(s_2)\in P\setminus\{0\}$ and $\lambda_G(s_1)^* \in P\setminus\{0\}$ imply $\delta_1\delta_2, \delta_1^{-1}\in \Sd(\varphi_G|_P)$. So we can consider $K=\Sd(\varphi_G|_P)^\perp$ and $(\ref{eqn:fixed_point_crossed_product_duality})$ implies $L(G)^K\leq P$. On the other hand, $P$ is generated by eigenoperators of $\sigma^{\varphi_G}$ with eigenvalues in $\Sd(\varphi_G|_P)$ (see \cite[Lemma 1.4]{GGLN25}). Since $\alpha$ extends $\sigma^{\varphi_G}$, it follows that for any eigenoperator $x\in P$ one has $\alpha_\gamma(x) = (\gamma\mid \delta) x$ for all $\gamma\in \Delta_G(G)\hat{\ }$ and some $\delta\in \Sd(\varphi_G|_P)$. But then $K= \Sd(\varphi_G|_P)^\perp$ implies $x\in L(G)^K$, and hence $P=L(G)^K$.
\end{proof}

\begin{thm}\label{thm:intermediate_algebras}
Let $G$ be a second countable almost unimodular group and let $\varphi_G$ be a Plancherel weight on $L(G)$. Assume that $G$ is non-unimodular. Then every intermediate von Neumann algebra $L(\ker{\Delta_G})\leq P\leq L(G)$ is of the form $P=L(H)$ for a closed intermediate group $\ker{\Delta_G}\leq H\leq G$ if and only if $L(\ker{\Delta_G})$ is a factor.
\end{thm}
\begin{proof}
\textbf{($\Rightarrow$):} We proceed by contrapositive and suppose $L(\ker{\Delta_G})$ is not a factor. Let $z\in L(\ker{\Delta_G})\cap L(\ker{\Delta_G})'$ be a projection which is neither zero nor one, and let $s\in G\setminus\ker{\Delta_G}$. Then the von Neumann algebra $P$ generated by $L(\ker{\Delta_G})$ and $z\lambda_G(s)$ is an intermediate von Neumann algebra $L(\ker{\Delta_G}) \leq P\leq L(G)$ but \emph{not} of the form $L(H)$ for any closed intermediate group $\ker{\Delta_G}\leq H\leq G$. Indeed, recall from the discussion at the beginning of this subsection that if $\alpha\colon \Delta_G(G)\hat{\ }\curvearrowright L(G)$ is the point modular extension of $\sigma^{\varphi_G}$, then it suffices to show $P\neq L(G)^K$ for any closed $K\leq \Delta_G(G)\hat{\ }$. Noting that $P$ is $\alpha$-invariant by virtue of $\alpha_\gamma( z\lambda_G(s)) = (\gamma\mid \Delta_G(s)) z \lambda_G(s)$ for all $\gamma\in \Delta_G(G)\hat{\ }$, we see that it suffices by Proposition~\ref{prop:characterizing_fixed_point_intermediate_algebras} to show $z_{\Delta_G(s)}\neq 1$. Rather, we claim that $z_{\Delta_G(s)}=z$. First note that
    \[
        z = (z\lambda_G(s)) (z\lambda_G(s))^* \leq z_{\Delta_G(s)}.
    \]
To see the other inequality, fix a partial isometry $v\in P$ satisfying $\alpha_\gamma(v) = (\gamma \mid \Delta_G(s)) v$ for all $\gamma\in \Delta_G(G)\hat{\ }$. By definition of $P$, $v$ can be approximated in the strong operator topology by a net $(p_i)_{i\in I}$ from the unital $*$-algebra generated by $L(\ker{\Delta_G})$ and $\{z\lambda_G(s), \lambda_G(s)^*z\}$. Viewing $p_i$ as a noncommutative polynomial in $\{z\lambda_G(s),\lambda_G(s)^*z\}$ with coefficients in $L(\ker{\Delta_G})$, we can write  $p_i=\sum_j m_{i,j}$ with each term of the form
    \[
        m_{i,j}= a_0 [z\lambda_G(s)]^{\epsilon_1} a_1\cdots a_{d-1} [z\lambda_G(s)]^{\epsilon_d} a_d
    \]
where $a_0,\ldots, a_d\in L(\ker{\Delta_G})$ and $\epsilon_1,\ldots, \epsilon_d \in \{1,*\}$. Define $\E_{\Delta_G(s)}(x):= \lambda_G(s)\E_{\varphi_G}(\lambda_G(s)^*x)$, where $\E_{\varphi_G}\colon L(G)\to L(\ker{\Delta_G})$ is the unique $\varphi_G$-preserving faithful normal conditional expectation. Then $v=\E_{\Delta_G(s)}(v)$, and so replacing each $p_i$ with $\E_{\Delta_G(s)}(p_i)$ we still have a net converging to $v$ in the strong operator topology. Moreover, noting that
    \begin{align*}
        \E_{\Delta_G(s)}&\left(a_0 [z\lambda_G(s)]^{\epsilon_1} a_1\cdots a_{d-1} [z\lambda_G(s)]^{\epsilon_d} a_d \right)\\
            &= \begin{cases}
            a_0 [z\lambda_G(s)]^{\epsilon_1} a_1\cdots a_{d-1} [z\lambda_G(s)]^{\epsilon_d} a_d &\text{if }|\{k\colon \epsilon_k=1\}| = |\{k\colon \epsilon_k=*\}|+1\\ 0 & \text{otherwise},
        \end{cases}
    \end{align*}
we see that this replacement merely deletes some terms $m_{i,j}$ in $p_i$ and preserves those for which there are exactly one more factor of $z\lambda_G(s)$ than of $\lambda_G(s)^*z$. Fix an $i\in I$ and a surviving term $m_{i,j}$, and let $1\leq k\leq d$ be the smallest index such that $|\{1\leq k'\leq k\colon \epsilon_{k'}=1\}| = |\{1\leq k'\leq k\colon \epsilon_{k'}=*\}|+1$; that is, reading $m_{i,j}$ from left to right the factor $(z\lambda_G(s))^{\epsilon_k}$ is the first time the number of factors of $z\lambda_G(s)$ exceeds the number of factors of $\lambda_G(s)^*z$. It follows that $\epsilon_k=1$ and $|\{1\leq k' \leq k-1\colon \epsilon_{k'}=1\}|= |\{1\leq k'\leq k-1\colon \epsilon_{k'}=*\}|$, and therefore the subword $a_0[z\lambda_G(s)]^{\epsilon_1}\cdots [z\lambda_G(s)]^{\epsilon_{k-1}} a_{k-1}$ lies in $L(G)^{\varphi_G}=L(\ker{\Delta_G})$. Recalling that $z$ lies in the center of $L(\ker{\Delta_G})$, we therefore have
    \begin{align*}
        m_{i,j} &= a_0[z\lambda_G(s)]^{\epsilon_1}\cdots [z\lambda_G(s)]^{\epsilon_{k-1}} a_{k-1} [z\lambda_G(s)] a_k \cdots (z\lambda_G(s))^{\epsilon_d} a_d\\
        &= z a_0[z\lambda_G(s)]^{\epsilon_1}\cdots [z\lambda_G(s)]^{\epsilon_{k-1}} a_{k-1} \lambda_G(s) a_k \cdots (z\lambda_G(s))^{\epsilon_d} a_d = z m_{i,j}.
    \end{align*}
Applying this to each term in $p_i$, we see that $p_i= zp_i$. Hence $v=zv$ as the strong operator topology limit of $(p_i)_{i\in I}$, and therefore $vv^* = z vv^* z \leq z$. This holds for every partial isometry in the definition of $z_{\Delta_G(s)}$ and so, combined with the previous inequality, we have $z_{\Delta_G(s)}=z$.\\

\noindent \textbf{($\Leftarrow$):} Suppose $L(\ker{\Delta_G})$ is a factor. Once again, let $\alpha\colon \Delta_G(G)\hat{\ }\curvearrowright L(G)$ be the point modular extension of $\sigma^{\varphi_G}$. Then the fixed point subalgebra is
    \[
        L(G)^{\Delta_G(G)\hat{\ }} = L(G)^{\varphi_G} = L(\ker{\Delta_G}),      
    \]
and the almost periodicity of $\varphi_G$ implies
    \[
         L(\ker{\Delta_G})' \cap L(G) = L(\ker{\Delta_G})'\cap L(\ker{\Delta_G})=\C
    \]
(see \cite[Theorem 10]{Con72}). Thus $\alpha$ is \emph{minimal} in the sense of \cite{ILP98}. Also observe that any intermediate von Neumann algebra $L(\ker{\Delta_G})\leq P\leq L(G)$ (including $P=L(G)$) is a factor since
    \[
        P'\cap P \subset (L(\ker{\Delta_G})'\cap L(G)=\C.
    \]
Finally, the second countability of $G$ implies $L(G)$ is separable (i.e. has a separable predual), and so we can apply \cite[Theorem 3.15]{ILP98} to see that every intermediate von Neumann algebra $P$ is of the form $L(G)^K$ for some closed $K\leq \Delta_G(G)\hat{\ }$. By the discussion at the beginning of this subsection, this implies $P=L(H)$ for the closed group $H=K_\perp \sltimes{(\beta,c)} \ker{\Delta_G} \leq G$.
\end{proof}

\subsection{Finite covolume subgroups and Murray--von Neumann dimension}

Recall that we say a closed subgroup $H\leq G$ of a locally compact group has \emph{finite covolume} if the quotient space $G/H$ admits a finite (non-trivial) $G$-invariant Radon measure $\mu_{G/H}$. It is always possible to normalize this measure in such a way that
    \begin{align}\label{eqn:quotient_Haar_measure_formula}
        \int_G f d\mu_G = \int_{G/H} \int_H f(st)\ d\mu_H(t) d\mu_{G/H}(sH) \qquad \qquad f\in L^1(G),
    \end{align}
where $\mu_G$ and $\mu_H$ are fixed left Haar measures on $G$ and $H$ respectively (see \cite[Theorem 2.49]{Fol16}). In this case, the \emph{covolume} of $(H,\mu_H) \leq (G,\mu_G)$ is the quantity
    \[
        [\mu_G:\mu_H ]:= \mu_{G/H}(G/H).
    \]
Note that when $H\leq G$ is a finite index inclusion and $\mu_H=\mu_G|_{\B(H)}$, then $\mu_{G/H}$ is the counting measure and one has $[\mu_G : \mu_H]=[G: H]$.

Suppose now that $G$ is an almost unimodular group and $H\leq G$ is a finite covolume subgroup. Finite covolume implies that $\Delta_G|_H= \Delta_H$ (see \cite[Theorem 2.49]{Fol16}), and hence $H$ is almost unimodular by Proposition~\ref{prop:almost_unimodular_subgroups}. Additionally, the openness of $\ker{\Delta_G}$ implies $(\ker{\Delta_G})H$ is an open subset of $G/H$.  A standard argument then implies $\mu_{G/H}((\ker{\Delta_G})H)>0$ (see \cite[Proposition 2.60]{Fol16}), which has two important consequences. The first is that $\ker{\Delta_H}$ is finite covolume in $\ker{\Delta_G}$ and the second is that $\Delta_H(H)=\Delta_G(H)$ is a finite index subgroup of $\Delta_G(G)$. Indeed, $\ker{\Delta_H} = \ker{\Delta_G}\cap H$ implies $\ker{\Delta_G}/\ker{\Delta_H}$ is homeomorphic to $(\ker{\Delta_G})H$, and the restriction of $\mu_{G/H}$ to this set is a non-trivial, finite, $(\ker{\Delta_G})$-invariant Radon measure. To see the second claim, fix a set of coset representatives $\Delta$ for $\Delta_G(G)/\Delta_H(H)$ and fix $s_\delta\in \Delta_G^{-1}(\{\delta\})$ for each $\delta\in \Delta$. Then one has
    \begin{align}\label{eqn:quotient_decomposition_over_kernel}
        G/H = \bigsqcup_{\delta\in \Delta} s_\delta (\ker{\Delta_G})H.
    \end{align}
The finiteness and $G$-invariance of $\mu_{G/H}$ therefore imply $|\Delta|<\infty$. We record these observations in the following proposition:

\begin{prop}\label{prop:finite_covolume_subgroups}
Let $G$ be an almost unimodular group with finite covolume subgroup $H\leq G$. Then:
    \begin{enumerate}[label=(\alph*)]
        \item $H$ is almost unimodular;
        \item $\Delta_G|_H=\Delta_H$;
        \item $\ker{\Delta_H}$ has finite covolume in $\ker{\Delta_G}$;
        \item $\Delta_H(H)$ is a finite index subgroup of $\Delta_G(G)$.
    \end{enumerate}
\end{prop}

The main goal of this subsection is to apply the Murray--von Neumann dimension theory for strictly semifinite weights developed in \cite{GGLN25} to Plancherel weights of almost unimodular groups. Recall from \cite{GGLN25} that an \emph{$(L(G),\varphi_G)$-module} is defined to be a pair $(\pi,\H)$ where $\pi\colon \<L(G),e_{\varphi_G}\> \to B(\H)$ is a normal unital $*$-homomorphism. For such a pair there always exists an ancillary Hilbert space $\K$ and an isometry $v\colon \H \to L^2(G)\otimes \K$ called a \emph{standard intertwiner} satisfying $v\pi(x) = (x\otimes 1)v$ for all $x\in \<L(G),\varphi_G\>$ (see \cite[Proposition 2.4]{GGLN25}). The \emph{Murray--von Neumann dimension} of $(\pi,\H)$ is defined as
    \[
        \dim_{(L(G),\varphi_G)}(\pi,\H):= (\varphi_G \otimes \Tr_{\K})\left[ (J_{\varphi_G}\otimes 1) vv^* (J_{\varphi_G}\otimes 1)\right],
    \]
and it is independent of $\K$ and $v$ (see \cite[Proposition 2.8]{GGLN25}). 

In the proof of the following theorem, it will once again be useful to recall that $e_{\varphi_G} = 1_{\ker{\Delta_G}}$ under the identification $L^2(L(G),\varphi_G)\cong L^2(G)$, and more generally 
    \[
        \lambda_G(s) e_{\varphi_G} \lambda_G(s)^* = 1_{\Delta_G^{-1}(\{\delta\})},
    \]
for any $s\in G$ with $\Delta_G(s)=\delta$. We will also implicitly use that for a closed subgroup $H\leq G$, $L(H)$ can be identified with $\{\lambda_G(t)\colon t\in H\}'' \leq L(G)$ (see, for example, \cite[Proposition 2.8]{HR19}).

\begin{thm}[{Theorem~\ref{introthm:C}}]\label{thm:covolume_dimension_formula}
Let $G$ be a second countable almost unimodular group with finite covolume subgroup $H\leq G$, and let $\varphi_G$ (resp. $\varphi_H$) be the Plancherel weight on $L(G)$ (resp. $L(H)$) associated to a left Haar measure $\mu_G$ on $G$ (resp. $\mu_H$ on $H$).
For each set $\Delta$ of coset representatives of $\Delta_H(H)\leq \Delta_G(G)$, there exists a unique injective, normal, unital $*$-homomorphism $\theta_\Delta\colon \<L(H),e_{\varphi_H}\> \to \<L(G),e_{\varphi_G}\>$ satisfying
    \[
        \theta_\Delta(\lambda_H(t))= \lambda_G(t)\quad t\in H, \qquad \text{ and } \qquad \theta_\Delta(e_{\varphi_H})=\sum_{\delta\in \Delta} 1_{\Delta_G^{-1}(\{\delta\})}.
    \]
Moreover, if $(\pi,\H)$ is a left $(L(G),\varphi_G)$-module, then $(\pi\circ \theta_\Delta,\H)$ is a left $(L(H),\varphi_H)$-module with
    \begin{align}\label{eqn:covolume_dimension_formula}
        \dim_{(L(H),\varphi_H)}(\pi\circ\theta_{\Delta},\H) = \left(\frac{1}{|\Delta|}\sum_{\delta\in \Delta} \delta\right) [\mu_G: \mu_H] \dim_{(L(G),\varphi_G)}(\pi,\H).
    \end{align}
\end{thm}
\begin{proof}
Throughout we denote $G_1:=\ker{\Delta_G}$ and $H_1:=\ker{\Delta_H}$. Fix a (necessarily finite) set $\Delta$ of coset representatives for $\Delta_H(H)\leq \Delta_G(G)$ (see Proposition~\ref{prop:finite_covolume_subgroups}), and set
    \[
        e:= \sum_{\delta\in \Delta} 1_{\Delta_G^{-1}(\{\delta\})}.
    \]
Note that $\Delta^{-1}=\{\delta^{-1}\colon \delta\in \Delta\}$ is also a set of coset representatives. By Proposition~\ref{prop:finite_covolume_subgroups}, $H_1$ has finite covolume in $G_1$, and so \cite[Lemma 1.1]{Mac52} provides a Borel section $\sigma_1 \colon G_1/H_1 \to G_1$. Fix $s_\delta\in \Delta_G^{-1}(\{\delta\})$ for each $\delta\in \Delta^{-1}$, then using (\ref{eqn:quotient_decomposition_over_kernel}) we can define a Borel section $\sigma\colon G/H\to G$ by
    \[
        \sigma(s_\delta s_1 H):= s_\delta \sigma_1(s_1 H_1) \qquad \qquad s_1\in G_1.
    \]
Suppose $\sigma(sH) = s t_s$ for some $t_s\in H$. Then
    \[
        \int_H f(\sigma(sH) t)\ d\mu_H(t) = \int_H f( s t_s t)\ d\mu_H(t) = \int_H f(s t)\ d\mu_H(t),
    \]
by the left invariance of $\mu_H$. Consequently, (\ref{eqn:quotient_Haar_measure_formula}) implies $w\colon L^2(G)\to L^2(H)\otimes L^2(G/H,\mu_{G/H})$ defined by
    \[
        [wf](t, sH) := f(\sigma(sH) t)
    \]
is a unitary with inverse determined by
    \[
        [w^*(g\otimes h)](s) = g(\sigma(sH)^{-1}s)h(sH),
    \]
for $g\in L^2(H)$, $h\in L^2(G/H,\mu_{G/H})$, and $s\in G$. Direct computations then show that for $t\in H$ and $\delta\in \Delta^{-1}$ one has
    \[
        w \rho_G(t) w^* = \rho_H(t)\otimes 1 \qquad \text{ and } \qquad w 1_{\Delta_G^{-1}(\{\delta\})} w^* = 1_{H_1}\otimes 1_{s_\delta G_1 H}.
    \]
It follows from the latter and the identity $J_{\varphi_G} 1_{\Delta_G^{-1}(\{\delta\})} J_{\varphi_G} = 1_{\Delta_G^{-1}(\{\delta^{-1}\})}$ that $wJ_{\varphi_G} e J_{\varphi_G} w^* = e_{\varphi_H}\otimes 1$. Consequently, if we set
    \[
        \theta_\Delta(x):= J_{\varphi_G} w^*[(J_{\varphi_H}xJ_{\varphi_H})\otimes 1]w J_{\varphi_G},
    \]
then $\theta_\Delta\colon \<L(H),e_{\varphi_H}\>\to \<L(G), e_{\varphi_G}\>$ is the desired $*$-homomorphism. 

Now, if $v\colon \H\to L^2(G)\otimes \K$ is a standard intertwiner for a left $(L(G),\varphi_G)$-module $(\pi,\H)$, then it follows that
    \[
        (J_{\varphi_H}\otimes 1\otimes 1)(w\otimes 1)(J_{\varphi_G}\otimes 1) v
    \]
is a standard intertwiner for the left $(L(H),\varphi_H)$-module $(\pi\circ \theta_\Delta, \H)$. Additionally, its Murray--von Neumann dimension is then given by
    \[
        \dim_{(L(H),\varphi_H)}(\pi\circ\theta_{\Delta},\H) = \left(\varphi_H\otimes \Tr_{L^2(G/H,\mu_{G/H})}\otimes \Tr_{\K}\right)\left[ (w\otimes 1)(J_{\varphi_G}\otimes 1)vv^* (J_{\varphi_G}\otimes 1) (w^*\otimes 1) \right],
    \]
Our strategy for relating the above quantity to the dimension of $(\pi,\H)$ as an $(L(G),\varphi_G)$-module will be to express conjugation by $w$ in terms of a faithful normal semifinite operator valued weight $T$ from $L(G)$ to $L(H)$ satisfying $\varphi_G=\varphi_H\circ T$. First note that such an operator valued weight exists by \cite[Theorem IX.4.18]{Tak03} since $H$ having finite covolume implies that $\Delta_G|_H=\Delta_H$ and therefore the modular automorphism group for $\varphi_G$ restricts to that of $\varphi_H$ on the copy of $L(H)$ inside $L(G)$. 

Toward relating $T$ and $w$, for any $h\in L^1(G/H,\mu_{G/H})$ let us denote
    \[
        \lambda_{G/H}(h):=\int_{G/H} \lambda_G(\sigma(sH)) h(sH)\ d\mu_{G/H}(sH) \in L(G).
    \]
For $g\in L^1(H)\cap L^2(H)$ and $h\in L^2(G/H,\mu_{G/H})$, a direct computation shows that $f:=w^*(g\otimes h)\in L^1(G)\cap L^2(G)$ with
    \begin{align}\label{eqn:left_regular_rep_of_w*}
        \lambda_G(f) = \lambda_{G/H}(h) \theta_{\Delta}(\lambda_H(g)).
    \end{align}
Using this, one has for $x\in L(G)$, $g_1,g_2\in L^1(H)\cap L^2(H)$, and $h_1,h_2\in L^2(G/H,\mu_{G/H})$ that
    \begin{align*}
        \< w x w^* g_1\otimes h_1, g_2\otimes h_2\>_{L^2(\mu_H)\otimes L^2(\mu_{G/H})} &= \varphi_G( \lambda_G( w^*(g_2\otimes h_2))^* x \lambda_G( w^*(g_1\otimes h_1))) \\
            &=\varphi_G( \theta_\Delta(\lambda_H(g_2))^* \lambda_{G/H}(h_2)^* x \lambda_{G/H}(h_1) \theta_\Delta(\lambda_H(g_1)) ) \\
            &= \varphi_H( \lambda_H(g_2)^* T[ \lambda_{G/H}(h_2)^* x \lambda_{G/H}(h_1)] \lambda_H(g_1))\\
            &= \< T[\lambda_{G/H}(h_2)^* x \lambda_{G/H}(h_1)] g_1, g_2\>_{L^2(\mu_H)}.
    \end{align*}
The density of $L^1(H)\cap L^2(H)$ in $L^2(H)$ therefore yields 
    \begin{align}\label{eqn:operator_valued_weight}
        1\otimes \omega_{h_1,h_2}(wxw^*) = T[\lambda_{G/H}(h_2)^* x \lambda_{G/H}(h_1)],
    \end{align}
where $\omega_{h_1,h_2} = \<\,\cdot\, h_1, h_2\>_{L^2(\mu_H)}$.

Now, from (\ref{eqn:quotient_decomposition_over_kernel}) we have
    \[
        L^2(G/H,\mu_{G/H}) = \bigoplus_{\delta\in \Delta^{-1}} L^2(s_\delta G_1 H, \mu_{G/H}),
    \]
and so if $\B_\delta \subset L^2(s_\delta G_1H, \mu_{G/H})$ is an orthonormal basis for each $\delta\in \Delta^{-1}$, then $\B:= \bigcup_{\delta\in \Delta^{-1}} \B_\delta$ is an orthonormal basis for $L^2(G/H,\mu_{G/H})$.  Given $f\in C_c(G)$ and $r\in G$, define $f^{(r)}(sH):=f(\sigma(sH) r)$ so that $f^{(r)}\in L^\infty(G/H,\mu_{G/H}) \subset L^2(G/H,\mu_{G/H})$ and for $b\in \mathcal{B}$ one has
    \[
        [\lambda_{G/H}(b)^* f](r) = \int_{G/H} \overline{b(sH)} f(\sigma(sH) r)\ d\mu_{G/H}(sH)  =  \< f^{(r)}, b\>_{L^2(\mu_{G/H})}.
    \]
Then for $f_1,f_2\in C_c(G)$ one has
    \begin{align*}
        \sum_{b\in \B_\delta} \< \lambda_{G/H}(b)^* f_1, \lambda_{G/H}(b)^* f_2\>_{L^2(\mu_G)} &= \sum_{b\in \B_\delta} \int_G \< f_1^{(r)}, b\>_{L^2(\mu_{G/H})} \< b, f_2^{(r)}\>_{L^2(\mu_{G/H})}\ d\mu_G(r) \\
        &= \int_G \<  f_1^{(r)}, 1_{s_\delta G_1 H} f_2^{(r)}\>_{L^2(\mu_{G/H})}\ d\mu_G(r)\\
        &= \int_{s_\delta G_1 H} \int_G f_1(\sigma(sH) r) \overline{f_2(\sigma(sH) r)}\ d\mu_G(r) d\mu_{G/H}(sH)\\
        &= \mu_{G/H}(s_\delta G_1 H) \< f_1, f_2\>_{L^2(\mu_G)},
    \end{align*}
where in the second equality we have used that $\B_\delta$ is an orthonormal basis for $1_{s_\delta G_1 H} L^2(G/H, \mu_{G/H}) = L^2(s_\delta G_1 H, \mu_{G/H})$. Noting that the covolume of $(H,\mu_H)\leq (G,\mu_G)$ is given by
    \[
        [\mu_G : \mu_H] = \sum_{\delta\in \Delta} \mu_{G/H}(s_\delta G_1 H) = \sum_{\delta\in \Delta} \mu_{G/H}(G_1 H) = |\Delta| \mu_{G/H}(G_1H),  
    \]
we see that the above computation shows
    \begin{align}\label{eqn:basis_to_covolume}
        \sum_{b\in \B_\delta} \lambda_{G/H}(b) \lambda_{G/H}(b)^* = \frac{1}{|\Delta|} [\mu_G : \mu_H]
    \end{align}
by the density of $C_c(G)$ in $L^2(G)$.

Combining (\ref{eqn:operator_valued_weight}) and (\ref{eqn:basis_to_covolume}), for $x\in L(G)^{\varphi_G}_+$ one then has
    \begin{align*}
        (\varphi_H\otimes \Tr_{L^2(G/H,\mu_{G/H})})(wxw^*) &= \sum_{b\in \B} \varphi_H(T[\lambda_{G/H}(b)^* x \lambda_{G/H}(b)])\\
            &= \sum_{b\in \B} \varphi_G(\lambda_{G/H}(b)^* x \lambda_{G/H}(b))\\
            &= \sum_{\delta\in \Delta}\sum_{b\in \B_\delta} \delta^{-1} \varphi_G( x^{1/2} \lambda_{G/H}(b)\lambda_{G/H}(b)^* x^{1/2})\\
            &=\left(\frac{1}{|\Delta|}\sum_{\delta\in \Delta^{-1}} \delta^{-1}\right) [\mu_G : \mu_H] \varphi_G(x) =\left(\frac{1}{|\Delta|}\sum_{\delta\in \Delta} \delta\right) [\mu_G : \mu_H] \varphi_G(x),
    \end{align*}
where in the third-to-last equality we have used that $\sigma_t^{\varphi_G}( \lambda_{G/H}(b)) = \delta^{it} \lambda_{G/H}(b)$ for $b\in \B_\delta$, which follows from $\supp(b)\subset s_\delta G_1 H = (\Delta_G\circ \sigma)^{-1}(\{\delta\})$. Consequently, we have
    \begin{align*}
         \dim_{(L(H),\varphi_H)}(\pi\circ\theta_{\Delta},\H) &= \left(\varphi_H\otimes \Tr_{L^2(G/H,\mu_{G/H})}\otimes \Tr_{\K}\right)\left[ (w\otimes 1)(J_{\varphi_G}\otimes 1)vv^* (J_{\varphi_G}\otimes 1) (w^*\otimes 1) \right]\\
            &= \left(\frac{1}{|\Delta|}\sum_{\delta\in \Delta} \delta\right) [\mu_G : \mu_H] \left(\varphi_G\otimes \Tr_{\K}\right)\left[ (J_{\varphi_G}\otimes 1)vv^* (J_{\varphi_G}\otimes 1)\right]\\
            &= \left(\frac{1}{|\Delta|}\sum_{\delta\in \Delta} \delta\right) [\mu_G : \mu_H] \dim_{(L(G),\varphi_G)}(\pi, \H),
    \end{align*}
as claimed.
\end{proof}

For $G,H$ as in the previous theorem, note that the restrictions $\mu_{\ker{\Delta_G}}:=\mu_G|_{\B(\ker{\Delta_G})}$ and $\mu_{\ker{\Delta_H}}:=\mu_H|_{\B(\ker{\Delta_H})}$ are left Haar measures satisfying
    \[
        [\mu_{\ker{\Delta_G}} : \mu_{\ker{\Delta_H}}] = \mu_{G/H}(\ker{\Delta_G} H) = \frac{1}{|\Delta|} [\mu_G : \mu_H].
    \]
Thus scaling factor in the previous theorem is equivalently given by
    \[
        \left(\frac{1}{|\Delta|}\sum_{\delta\in \Delta} \delta\right) [\mu_G : \mu_H] = \left(\sum_{\delta\in \Delta} \delta\right) [\mu_{\ker{\Delta_G}} : \mu_{\ker{\Delta_H}}]
    \]
Also observe that if $\Delta_H(H)$ is dense in $\R$, then $\frac{1}{|\Delta|}\sum_{\delta\in \Delta} \delta$ can be made arbitrarily close to one by choosing the coset representatives in $\Delta$ sufficiently close to one.
    
In the case that $\Delta_H(H)=\Delta_G(H)$, one can choose $\Delta=\{1\}$ and the above scaling factor is simply $[\mu_G: \mu_H]$. Additionally, one has $\theta_\Delta(e_{\varphi_H}) = e_{\varphi_G}$. This yields the following corollary.

\begin{cor}\label{cor:covolume_dimension_formula}
Let $G$ be a second countable almost unimodular group with finite covolume subgroup $H\leq G$, and let $\varphi_G$ (resp. $\varphi_H$) be the Plancherel weight on $L(G)$ (resp. $L(H))$ associated to a left Haar measure $\mu_G$ on $G$ (resp. $\mu_H$ on $H$). Suppose $\Delta_H(H)=\Delta_G(G)$. Then identifying $\<L(H),e_{\varphi_H}\>\cong \<L(H),e_{\varphi_G}\> \leq \< L(G), e_{\varphi_G}\>$ one has
    \begin{align}\label{eqn:covolume_dimension_formula_equal_point_modular_groups}
        \dim_{(L(H),\varphi_H)}(\pi, \H)= [\mu_G : \mu_H] \dim_{(L(G),\varphi_G)}(\pi,\H)
    \end{align}
for any left $(L(G),\varphi_G)$-module $(\pi, \H)$.
\end{cor}

\begin{rem}
In the context of Theorem~\ref{thm:covolume_dimension_formula}, suppose $G/H$ is a finite set. Then as noted above one has $[\mu_G : \mu_H]= [G : H]$ when $\mu_H= \mu_G|_{\B(H)}$. In this case, (\ref{eqn:covolume_dimension_formula_equal_point_modular_groups}) should be compared with \cite[Proposition 3.11]{GGLN25}, which established the same formula for finite index inclusions of separable factors equipped with almost periodic states. $\hfill\blacksquare$
\end{rem}

\begin{rem}
If $H\leq G$ is an inclusion of locally compact groups such that $\Delta_G|_H= \Delta_H$ (equivalently, $G/H$ admits a non-trivial $G$-invariant Radon measure), then by the same argument as in the proof of Theorem~\ref{thm:covolume_dimension_formula} there exists a faithful normal operator valued weight $T$ from $L(G)$ to $L(H)$ that intertwines a fixed pair of Plancherel weights $\varphi_G,\varphi_H$. In the special case that $H$ has finite covolume, the inclusion $L(H)\leq L(G)$ is \emph{compact} is the sense of \cite{HO89} where the conditional expectation from $(J_{\varphi_G}L(H) J_{\varphi_G})' = R(H)'$ (the basic construction for $L(H)\leq L(G)$) onto $L(G)$ is given by
    \[
        \mathcal{E}(x) = \frac{1}{\mu_{G/H}(G/H)} \int_{G/H} \rho_G(s) x \rho_G(s)^*\ d\mu_{G/H}(sH) \qquad x\in R(H)'.
    \]
Furthermore, the set $\{\lambda_{G/H}(b)\colon b\in \mathcal{B}\}$ in the proof of Theorem~\ref{thm:covolume_dimension_formula} is a Pimsner--Popa basis relative to $T$. Indeed, for $g\in L^1(H)\cap L^2(H)$ and $h\in L^2(G/H,\mu_{G/H})$ one has
    \[
        T\left[\lambda_{G/H}(b)^* \lambda_G(w^*(g\otimes h))\right] = T\left[\lambda_{G/H}(b)^* \lambda_{G/H}(h) \right] \lambda_H(g)= \<h, b\>_{L^2(\mu_{G/H})} \lambda_H(g).
    \]
Using that $\|\lambda_{G/H}(k)\| \leq \|k\|_{L^1(\mu_{G/H})} \leq \|k\|_{L^2(\mu_{G/H})} [\mu_G : \mu_H]^{\frac12}$ for any $k\in L^2(G/H,\mu_{G/H})$, it follows that
    \[
        \lambda_G(w^*(g\otimes h)) = \lambda_{G/H}(h) \lambda_H(g) = \sum_{b\in B} \lambda_{G/H}(b) \< h,b\>_{L^2(\mu_{G/H})} \lambda_H(g) = \sum_{b\in B} \lambda_{G/H}(b) T \left[ \lambda_{G/H}(b)^* \lambda_{G}(w^*(g\otimes h)) \right],
    \]
where the sum converges in norm and we have identified $\theta_\Delta(\lambda_H(g)) = \lambda_H(g)$. $\hfill\blacksquare$
\end{rem}

In \cite{AS77}, Atiyah and Schmid provided a formula relating covolumes of lattices, formal degrees of irreducible square integrable representations, and Murray--von Neumann dimension. More precisely, suppose $G$ is a locally compact group with a lattice subgroup $\Gamma\leq G$. Then $G$ is necessarily unimodular (see, for example, \cite[Proposition B.2.2.(ii)]{BdlHV08}), and the Plancherel weight on $L(\Gamma)$ is a faithful normal tracial state $\tau_\Gamma$. If $(\pi, \H)$ is an irreducible square integrable representation of $G$, then $\pi$ has a unique extension $\widetilde{\pi}\colon L(G)\to B(\H)$. It was shown in \cite[Equation (3.3)]{AS77} (see also \cite[Theorem 3.3.2]{GdlHJ89}) that one has
    \[
        \dim_{(L(\Gamma),\tau_\Gamma)}(\widetilde{\pi}|_{L(\Gamma)}, \H) = d_\pi [\mu_G : \mu_\Gamma],
    \]
where the formal degree $d_\pi$ of $(\pi, \H)$ is with respect to some fixed left Haar measure $\mu_G$ on $G$, and $\mu_\Gamma$ is the counting measure. We generalize this formula to finite covolume subgroups of almost unimodular groups in Theorem~\ref{thm:formula} below. 

Recall from \cite[Proposition 7]{DM76} that an irreducible square integrable representation $(\pi,\H)$ of an almost unimodular group $G$ is induced by an irreducible square integrable representation $(\pi_1,\H_1)$ of its unimodular part $\ker{\Delta_G}$. Additionally, the square integrability implies these representations admit extensions to $L(G)$ and $L(\ker{\Delta_G})$, respectively. Thus, by Theorems~\ref{thm:reps_induced_from_kernel} and \ref{thm:extending_reps_to_basic_construction}, there is a further extension $\widetilde{\pi}$ of $\pi$ to the basic construction $\<L(G),e_{\varphi_G}\>$ satisfying $\widetilde{\pi}(e_{\varphi_G})\H= \H_1$.

\begin{thm}[{Theorem~\ref{introthm:D}}]\label{thm:formula}
Let $G$ be a second countable almost unimodular group with finite covolume subgroup $H\leq G$, let $\varphi_G$ (resp. $\varphi_H$) be the Plancherel weight on $L(G)$ (resp. $L(H)$) associated to a left Haar measure $\mu_G$ on $G$ (resp. $\mu_H$ on $H$), and for each set $\Delta$ of coset representatives of $\Delta_H(H)\leq \Delta_G(G)$ let $\theta_\Delta\colon \<L(H),e_{\varphi_H}\> \to \<L(G),e_{\varphi_G}\>$ be as in Theorem~\ref{thm:covolume_dimension_formula}. Let $(\pi, \H)$ be an irreducible square integrable representation of $G$, let $(\pi_1,\H_1)$ be the irreducible square integrable representation of $\ker{\Delta_G}$ that induces $(\pi,\H)$, and let $(\widetilde{\pi},\H)$ be the representation of $\<L(G),e_{\varphi_G}\>$ extending $(\pi,\H)$. Then one has
    \[
        \dim_{(L(H), \varphi_H)}(\widetilde{\pi}\circ\theta_\Delta, \H) = d_{\pi_1} \left(\frac{1}{|\Delta|}\sum_{\delta\in \Delta} \delta\right) [\mu_G : \mu_H],
    \]
where $d_{\pi_1}$ is the formal degree of $(\pi_1, \H_1)$ with respect to $\mu_{\ker{\Delta_G}}:=\mu_G|_{\mathcal{B}(\ker{\Delta_G})}$.
\end{thm}
\begin{proof}
By Theorem~\ref{thm:covolume_dimension_formula} it suffices to consider the case when $H=G$. As usual, we denote $G_1:=\ker{\Delta_G}$. The strict semifiniteness of $\varphi_G$ implies that we can find a family of projections $\{p_i\in L(G_1)\cap \dom(\varphi_G)\}$ satisfying $\sum_i p_i =1$. Since $p_i\in \dom(\varphi_G)$ for each $i\in I$, it follows that $p_i=\lambda_G(\xi_i)$ for some left convolver $\xi_i\in L^2(G_1,\mu_G)$. We claim that for any $\eta\in L^2(G)$ one has
    \begin{align}\label{eqn:semicyclic_presentation}
        \|\eta\|_{L^2(\mu_G)}^2 = \sum_{i\in I} \int_G |\<\eta, \lambda_G(s) \xi_i\>_{L^2(\mu_G)}|^2\ d\mu_G(s).
    \end{align}
First, consider the case when $\eta = J_{\varphi_G} p_i J_{\varphi_G} \eta$. Observe that
    \begin{align*}
        \eta(s) &= [J_{\varphi_G} \lambda_G(\xi_i) J_{\varphi_G} \eta](s) = \Delta_G(s)^{-1/2} \overline{[\lambda_G(\xi_i) J_{\varphi_G} \eta](s^{-1})}\\
            &= \int_{G_1} \Delta_G(s)^{-1/2} \overline{\xi_i(r)} \overline{[J_{\varphi_G}\eta](r^{-1}s^{-1})}\ d\mu_G(r) = \int_{G_1} \overline{\xi_i(r)} \eta(sr)\ d\mu_G(r) = \< \eta, \lambda_G(s) \xi_i\>_{L^2(\mu_G)}.
    \end{align*}
Consequently, one has
    \[
        \|\eta\|_{L^2(\mu_G)}^2 = \int_G |\eta(s)|^2\ d\mu_G(s) = \int_G |\< \eta, \lambda_G(s)\xi_i\>_{L^2(\mu_G)}|^2\ d\mu_G(s).
    \]
Then general case follows from $\sum_i J_{\varphi_G} p_i J_{\varphi_G}=1$.

Now, let $v\colon \H\to L^2(G)$ be a standard intertwiner. Then
    \[
        \dim_{(L(G),\varphi_G)}(\widetilde{\pi},\H) = \varphi_G( J_{\varphi_G} vv^* J_{\varphi_G}) = \sum_{i\in I} \varphi_G( p_i J_{\varphi_G} vv^* J_{\varphi_G} p_i) = \sum_{i\in I} \| v^* J_{\varphi_G} \xi_i\| = \sum_{i\in I} \|v^* \xi_i\|.
    \]
If $D$ is the formal degree operator of $(\pi,\H)$, then $\xi_i$ being supported in $G_1$ implies $D^{-1/2} v^* \xi_i = d_{\pi_1}^{1/2} v^* \xi_i$ (see Theorem~\ref{thm:sq_int_irreps}). Using this along with (\ref{eqn:formal_degree_operator_equation}) and (\ref{eqn:semicyclic_presentation}), for any unit vector $\eta\in \H$ we have
    \begin{align*}
        \dim_{(L(G),\varphi_G)}(\widetilde{\pi},\H) &= d_{\pi_1} \sum_{i\in I} \int_G |\<\pi(s) v^*\xi_i, \eta\>|^2\ d\mu_G(s)\\
            &= d_{\pi_1} \sum_{i\in I} \int_G |\< \lambda_G(s) \xi_i, v\eta\>_{L^2(\mu_G)}|^2\ d\mu_G(s) = d_{\pi_1} \|v\eta\|_{L^2(\mu_G)}^2 = d_{\pi_1},
    \end{align*}
as claimed.
\end{proof}

Finally, we highlight the special case $\Delta_H(H)=\Delta_G(G)$ in the previous theorem as a corollary:

\begin{cor}
Let $G$ be a second countable almost unimodular group with finite covolume subgroup $H\leq G$, and let $\varphi_G$ (resp. $\varphi_H$) be the Plancherel weight on $L(G)$ (resp. $L(H)$) associated to a left Haar measure $\mu_G$ on $G$ (resp. $\mu_H$ on $H$). Suppose $\Delta_H(H)=\Delta_G(G)$ so that we may identify $\<L(H),e_{\varphi_H}\>\cong \<L(H),e_{\varphi_G}\> \leq \<L(G),e_{\varphi_G}\>$. Let $(\pi, \H)$ be an irreducible square integrable representation of $G$, let $(\pi_1,\H_1)$ be the irreducible square integrable representation of $\ker{\Delta_G}$ that induces $(\pi,\H)$, and let $(\widetilde{\pi},\H)$ be the representation of $\<L(G),e_{\varphi_G}\>$ extending $(\pi,\H)$. Then one has
    \[
        \dim_{(L(H), \varphi_H)}(\widetilde{\pi}, \H) = d_{\pi_1} [\mu_G : \mu_H],
    \]
where $d_{\pi_1}$ is the formal degree of $(\pi_1, \H_1)$ with respect to $\mu_{\ker{\Delta_G}}:=\mu_G|_{\mathcal{B}(\ker{\Delta_G})}$.
\end{cor}

%%%%%%%%%%%%%%%%%%%%%%%%%
%       References      %
%%%%%%%%%%%%%%%%%%%%%%%%%

\bibliographystyle{amsalpha}
\bibliography{references}

\end{document}